\documentclass[reqno,12pt]{amsart}
\usepackage{xspace,cancel}
\usepackage[normalem]{ulem}
\usepackage[bookmarksnumbered,colorlinks]{hyperref}
\usepackage{graphics}
\usepackage{enumitem}
\usepackage[T1]{fontenc}
\usepackage{amsmath}
\usepackage{url}

\hypersetup{
	pdftoolbar=true,
	pdfmenubar=true,
	pdffitwindow=false,
	pdfstartview={FitH},
	pdftitle={},
	pdfauthor={Erasmo Caponio, Dario Corona},
	pdfsubject={},
	pdfkeywords={},
	pdfnewwindow=true,
	colorlinks=true, 
	linkcolor=blue,
	citecolor=blue,
	urlcolor=blue,
}

\newcommand{\norm}[1]{\lVert #1\rVert}

\newcommand{\abs}[1]{|#1|}
\newcommand{\diff}{\text{d}}
\newcommand{\cat}{\text{cat}}

\theoremstyle{plain}\newtheorem{teo}{Theorem}[section]
\theoremstyle{plain}\newtheorem{proposition}[teo]{Proposition}
\theoremstyle{plain}\newtheorem{lemma}[teo]{Lemma}
\theoremstyle{plain}\newtheorem{corollary}[teo]{Corollary}

\theoremstyle{definition}\newtheorem{definition}[teo]{Definition}
\theoremstyle{definition}\newtheorem{example}[teo]{Example}
\theoremstyle{definition}\newtheorem{assumption}{Assumption}

\theoremstyle{remark}\newtheorem{remark}[teo]{Remark}

\numberwithin{teo}{section} 
\numberwithin{equation}{section}

\usepackage[colorinlistoftodos,textsize=footnotesize]{todonotes}

\newcounter{mycomment}

\newcommand{\beq}{\begin{equation}}
\newcommand{\eeq}{\end{equation}}
\newcommand{\bere}{\begin{remark}}
\newcommand{\ere}{\end{remark}}
\newcommand{\bpr}{\begin{proposition}}
\newcommand{\epr}{\end{proposition}}
\def\bal#1\eal{\begin{align}#1\end{align}}       
\def\baln#1\ealn{\begin{align*}#1\end{align*}}     
\def\bml#1\eml{\begin{multline}#1\end{multline}}    
\def\bmln#1\emln{\begin{multline*}#1\end{multline*}} 
\def\bga#1\ega{\begin{gather}#1\end{gather}}
\def\bgan#1\egan{\begin{gather*}#1\end{gather*}}

\newcommand{\R}{\mathbb R}

\newcommand{\J}{\mathcal{J}}
\newcommand{\A}{\mathcal{A}}

\newcommand{\D}{\mathcal{D}}

\newcommand{\de}{\mathrm{d}}
\newcommand{\inte}{\int_0^1}
\newlist{list1}{enumerate}{1}
\setlist[list1,1]{label={(\roman{list1i})},ref=\roman{list1i}}

\title[Indefinite Lagrangians with an affine Noether charge]{A variational setting for \\  an indefinite Lagrangian with \\ an affine Noether charge
}
\author[E. Caponio]{Erasmo Caponio}
\address{Dipartimento di Meccanica, Matematica e Management, Politecnico di Bari, Bari, Italy}
\email{erasmo.caponio@poliba.it}
\thanks{E. Caponio is partially supported by PRIN 2017JPCAPN {\em Qualitative and quantitative aspects of nonlinear PDEs.}}

\author[D. Corona]{Dario Corona}
\address{Mathematics Division, School of Science and Technology, University of Camerino, Camerino, Italy}
\email{dario.corona@unicam.it}
\thanks{Both authors thank the partial support of GNAMPA INdAM – Italian National Institute of High Mathematics.}


\begin{document}
\begin{abstract}
	We introduce a variational setting for the action functional of an
	autonomous and indefinite Lagrangian on a finite dimensional manifold $M$.
	Our basic  assumption is the existence of an infinitesimal symmetry
	whose Noether charge is the sum of a  one-form and a function on $M$.
	Our setting includes different types of Lorentz-Finsler Lagrangians
	admitting a timelike Killing vector field. 
\end{abstract}
\maketitle

\section{Introduction}
The principle of least or stationary-action in Lagrangian mechanics has been at
the heart of the development of the variational
calculus. It has given rise to different methods for solving 
the problem of finding (or at least establishing the existence of) a path of
evolution between two points of a dynamical system described by a finite
number of variables (see, e.g., \cite{Bolza04, Morse28, Tonell25}).
The techniques  developed to get solutions have been proved to be useful in the
study of general Lagrangian systems with an infinite number of degrees of
freedom (see, e.g., \cite{Fonseca07,Struwe08}).
A very classical  field of application of  these methods is  the geodesic
problem  in Riemannian and Finsler geometry. In this case,  completeness of
the metric is enough to get a solution with fixed end points and topological
arguments give  multiplicity of geodesics.
The landscape is quite different for the analogous problem on a Lorentzian
manifold where  (geodesic) completeness is not enough to get compactness
properties on the space of paths between two points and other geometric
conditions as global hyperbolicity have been considered as a replacement
\cite{avez1963,seifert1967}.
Only recently the features underlying global hyperbolicity, in connection with
the geodesic problem and more generally with causality, start to find a field
of applications beyond classical Lorentzian geometry
(see \cite{
	BerSuh18,
	FatSic12,
	JavSan20,
Minguz19}).

On the other hand, the existence of a symmetry that leaves invariant the action
functional of a Lagrangian is a source of information about its stationary
points through the Noether's theorem.
The impact of this result in variational calculus can be hardly overestimated.
A nice  application of it  to the geodesic problem of a Lorentzian manifold can be found in \cite{giannoni1999}, where  a stationary spacetime $(M,g_L)$ 
(i.e. a spacetime endowed with a  timelike Killing vector field) is
considered.
In this case, the Noether charge associated to the  Killing field
is used to get a reduction of the Sobolev manifold of paths between two points
$p$ and $q$ in $M$, where the energy functional of the Lorentz metric is defined,
to the infinite dimensional submanifold $\mathcal N_{p,q}$
of the curves with a.e. constant Noether charge.
This reduction resembles the classical Routh reduction for Lagrangian systems
(see, e.g. \cite{CraMes08, MaRaSc00})
but it involves merely the paths space and not the phase space. 
The roots of the idea of this infinite dimensional reduction are in a couple of
papers about geodesic connectedness of  static and stationary spacetimes
admitting a global splitting \cite{benci1991, GiannoniMasiello1991} and,
indeed, some local computation in \cite{giannoni1999} and in the present paper
(see Theorem~\ref{teo:cprecbound}) are based on those papers.

Our  goal is to show that the full variational setting in \cite{giannoni1999}
admits a generalization for an indefinite $C^1$ Lagrangian $L$
on a smooth finite dimensional manifold $M$.
We assume that $L$ is invariant by a one-parameter group of local
diffeomorphisms whose infinitesimal generator is a vector field $K$ and that
the associated Noether charge is a $C^1$ function $N$ on $TM$, which is affine in each tangent space $T_xM$:
\begin{equation}
	\label{N}
	N(x,v)=Q(v)+d(x),
\end{equation}
where $Q$ and $d$ are  a one-form and a function on $M$, respectively. We assume also that $d$ is invariant by the flow of $K$ and  
$Q(K)<0$ (see Assumption~\ref{ass:L}). 
Notice that in the case of a stationary Lorentzian manifold, $d=0$ and $Q$ coincides with the one-form metrically equivalent to the timelike Killing field $K$.

In Theorems~\ref{teo:main} and \ref{teo:multiplicity} we obtain
existence and multiplicity of weak solutions to the Euler-Lagrange
equation of the action functional of $L$ connecting two given points on $M$.
The regularity of solutions is analysed in Appendix~\ref{sec:regularity}.  
A key assumption in Theorem~\ref{teo:main} is {\em $c$-boundedness}
(Definition~\ref{c-boundedness}) of $\mathcal N_{p,q}$.
Under conditions contained in Assumptions~\ref{ass:L}--\ref{ass:bounds},
$c$-boundedness implies that the reduced action functional $\J$
(differently from the action)
is bounded from below (Proposition~\ref{prop:Jbounded})
and satisfies the Palais-Smale condition (Theorem~\ref{PS}).
We show in Section~\ref{pseudocoercivity} that $c$-boundedness is
essentially equivalent to {\em $c$-precompactness} of $\mathcal N_{p,q}$,
a condition introduced in \cite{giannoni1999}
which is a compactness property of the set of paths in a
sublevel of the reduced action functional.
Actually, on a stationary Lorentzian manifold $M$,
if $\mathcal N_{p,q}$ is $c$-precompact for all $c\in\R$
then $M$ is globally hyperbolic (see \cite[Proposition B1]{giannoni1999}
in the case when the timelike Killing vector field is complete and
\cite[Section 6.4-(a)]{CaFlSa08} for any timelike Killing vector field).
On the converse, if $M$ is globally hyperbolic with a complete smooth
Cauchy hypersurface then $\mathcal N_{p,q}$ is $c$-precompact for all $c\in\R$
(see \cite[Theorem 5.1]{CaFlSa08}).
Thus, if $c$-precompactness is satisfied for all $c\in\R$,
the spacetime $M$ cannot be compact.
Inspired by Proposition A.3 in \cite{giannoni1999},
we  give a condition that implies $c$-precompactness of $\mathcal N_{p,q}$,
for all $c\in \R$ and all $p,q\in M$, in our setting,  and that cannot be satisfied if $M$ is
compact (see Proposition~\ref{cbounded}).

The  Lagrangians that we consider (see Section~\ref{sec:examples}) include,
but are not limited to, $C^1$ stationary Lorentzian metrics,
electromagnetic type Lagrangians on a stationary Lorentzian manifold with a
Killing vector field $K$ and $K$-invariant potentials (see, e.g.
\cite{Bartol99, CapMas02, CaMaPi04, Vitori20}) and
some stationary Lorentz-Finsler metrics.
Loosing speaking, a Lorentz-Finsler metric  is an  indefinite Lagrangian,
positively homogeneous of degree two in the velocities,
that generalizes the quadratic form of a Lorentzian metric in the same way as 
the square of a Finsler metric generalizes the square of the norm of a Riemannian metric.
They were studied  by K. Beem \cite{Beem70} following some work by H. Busemann.
Although considered from time to time in works about anisotropy in
special and general relativity
(even if  they often  appear as the square of a more
fundamental function, positively homogeneous of degree one in the velocities,
see e.g. \cite{Brandt92,Horvat58,Rutz93}),
there has been a growing interest about them
(or their possible generalizations as non-degenerate Lagrangians defined
on a cone bundle on $M$)
in the  last decade, see for example 
\cite{
	AazJav16,
	BeJaSa20,
	CapMas20,
	GaPiVi12,
	HasPer19,
	HoPfVo19,
	LaPeHa12,
	LuMiOh21,
	Minguz15,
Perlic06}.

Some explicit examples that are covered by our present setting   are
Beem's Lorentz-Finsler metrics endowed with a timelike Killing vector field $K$
including also  their sum with  a potential function and a one-form, both
invariant by the flow of $K$ (see Example~\ref{beemxample}).
In particular, this class includes Lagrangians defined as
\[
	L=F^2-\omega^2,
\]
introduced in \cite{JavSan20},
where $F$ and $\omega$ are, respectively,
a Finsler metric and a one-form on $M$ both invariant
by the one-parameter group of local diffeomorphisms
generated by $K$, provided  a sign  assumption on $F^2(K)-\omega^2(K)$  is satisfied
(see Example~\ref{JS}).
Other examples are given by
Lagrangians $L$ that
locally, i.e. on a  neighborhood of the type $S\times (a,b)\subset M$,
can be expressed as
\begin{equation}
	\label{CS}
	L= L_0+2(\omega+d/2) \de t -\beta\de t^ 2,
\end{equation}
where $L_0$ is a $C^1$ Tonelli Lagrangian on $S$, with quadratic growth in the velocities,
$\omega$, $d$  and $\beta$ are respectively a $C^1$ one-form on $S$
and two $C^1$ functions on $S$ with $\beta>0$
(see Example~\ref{ex:general-ind-Lagrangian} and Proposition~\ref{productype}).
We include the possibility that the Lagrangian $L_0$
might not be twice differentiable on the zero section of $TS$,
but we require that it is pointwise strongly convex
(see Assumption~\ref{ass:L1}-\eqref{monotonehyp}).
Notice that the possible lack of twice differentiability of $L_0$
at the zero section implies that $L$
is not twice differentiable along the line bundle defined by $K=\partial_t$,
being $t$ the natural coordinate on the interval $(a,b)$.
Lagrangians of the type \eqref{CS} on a global splitting $S\times\R$,
with $L_0$ being the square of a Finsler metric and $d=0$,
were introduced in \cite{LaPeHa12} when $\omega= 0$
(see also \cite{CapSta16}) and in \cite{caponio2018} for $\omega\neq 0$.

Let us point out a comment about the regularity of the objects
we consider in this work.
We consider a smooth, finite dimensional manifold $M$;
the Lagrangian $L$ and the vector field $K$ are of class $C^1$ on $TM$.
Lorentz-Finsler Lagrangians are  
not twice differentiable at the zero section of $TM$,
hence assuming that $L$ is  $C^1$
is  motivated by that wide class of indefinite Lagrangians. 
We are confident  that both the regularity of $L$
and the linearity  of the Noether charge can be further relaxed  at
least for the existence of a global  minimizer of the reduced action
functional.
This is clearly suggested by the fact that $L$ is the sum of a
Lagrangian which is strongly convex in the velocities and a $C^1$ Lagrangian
related to the Noether charge (see \eqref{L}),
and that some computations of this work are more related
to the sublinearity of the Noether charge than to its expression \eqref{N}.

\section{Notations, assumptions and preliminary results}
Let $M$ be a smooth, connected, $(m+1)$-dimensional manifold,
with $m \ge 1$;
let us denote by $TM$ the tangent bundle of $M$.
Throughout the paper, we  consider a (auxiliary) complete Riemannian metric $g$ on $M$ and we denote by 
$\norm{\cdot}\colon TM \to \mathbb{R}$ its induced norm, i.e.
$\norm{v}^2 = g(v,v)$ for all $v \in TM$. 

We will often denote an element of $TM$ as a couple $(x,v)$,
$x\in M$, $v\in T_xM$
(for example we use such a  notation in connection with the variables of
an autonomous Lagrangian $L:TM\to \R$,  i.e. we will write $L=L(x,v)$).
On the other hand,
we will avoid specifying the point $x$ where a one-form $\omega$ or
a vector field $K$ on $M$ is applied, and we will write,
for example $\omega(v)$, $v\in TM$ or also $\omega(K)$.
Some exceptions are possible for the sake of clarity, and we will write then,
e.g., $K_x$ or $\omega_x(v)$, $v\in T_xM$ and also $\omega_x(K)$.
We will  often explicitly write the variable of a function on $M$,
like  in  $d(x)$, $C(x)$, $\lambda(x)$, etc.
When a vector field $K$ on $M$ is evaluated along a curve $z:[0,1]\to M$,
we will write $K(z)$.
In some cases we will look at a one-form $Q$ on $M$ also as
a function on $TM$ writing then $Q(x,v)$.

Let $L\colon TM\to\R$ be a Lagrangian on $M$.
For any $(x,v)\in TM$, we denote by $\partial_v L(x,v)[\cdot]$
the vertical derivative of $L$, 
i.e. for all $x\in M$ and all $v,w\in T_xM$
\[
\partial_v L(x,v)[w]:=\frac{\de}{\de s}L(x,v+sw)|_{s=0}.
\]
We need also a notion of horizontal derivative of the Lagrangian $L$
(a derivative w.r.t. $x$).
Let $(x^0,\dots,x^m)$ be coordinates on $M$
and let $(x^0,\dots,x^m, v^{0},\dots,v^m)$ be the induced ones on $TM$.
Let $(x,v)\in TM$, with coordinates values  $(x^0,\dots,x^m, v^{0},\dots,v^m)$;
we define  $\partial_x L(x,v)[\cdot]$ as the $v$-depending one-form
on $M$ locally given by 
\[
\partial_x L(x,v)[w]:=\sum_{i=0}^m\frac{\partial L}{\partial x^i}(x,v)w^i.
\]

\begin{remark}
\label{atlas}
Even though, differently from the vertical derivative,
this definition is not intrinsic, it  fits our purposes
(in the following, we will make extensively use of local arguments
in computations involving $L$).
In particular, we denote by
$\| \partial_x L_c(x,v)\|$ and $\| \partial_v L_c(x,v)\|$
the two scalar fields on $TM$
which are pointwise the
norm of the above two linear operators w.r.t. $g$.
\end{remark}

\begin{assumption}
	\label{ass:L}
	The Lagrangian $L\colon TM \to \mathbb{R}$ satisfies the following conditions:
	\begin{list1}
	\item\label{c1}$L\in C^1(TM)$;
	\item\label{linear}
		there exists a $C^1$ vector field $K$ on $M$ such that 
		$L$ is invariant by the one-parameter group of
		local $C^1$ diffeomorphisms generated by $K$
		(we call $K$ an {\em infinitesimal symmetry of $L$});
		moreover the {\em  Noether charge},
		i.e. the map $(x,v)\in TM\mapsto \partial_vL(x,v)[K]\in\mathbb{R}$,
		is a function $N$ on $TM$ which is the sum of
		a $C^1$ one-form $Q$ on $M$
		and a $C^1$ function $d\colon M\to\R$, i.e.
		\begin{equation}
			\label{noether}
				N(x,v) :=\partial_vL(x,v)[K]=Q(v)+d(x);
		\end{equation}
	\item\label{negativeQKhyp}the function $d$ in \eqref{noether} is invariant by the flow of $K$ (in particular the case when $d$ is a constant function is compatible); moreover,
		\begin{equation}
			\label{negativeQK}
			Q(K)<0.
		\end{equation}
	\end{list1}
\end{assumption}

\begin{remark}
	Vector fields $K$ which are infinitesimal symmetries for $L$
	can be characterized similarly to Killing vector fields for Finsler metrics
	(see, e.g., \cite{caponio2018}).
	We denote by $K^c$ the \textit{complete lift} of  $K$ to $TM$,
	which,   using Einstein summation convention,
	is locally defined as:	
	\begin{equation}
		\label{eq:lift-LocCoor}
		(K^c)_{(x,v)}=K^h(x)\frac{\partial }{\partial x^h} 
		+ \frac{\partial K^h}{\partial x^i}(x)v^i \frac{\partial }{\partial v^h}.
	\end{equation}
	It follows that,
	if $\psi$ is a local flow of $K$,
	then 
	for any $(x,v)\in TM$ the local flow $\psi^c$ of $K^c$ on $TM$ is given by
	$\psi^c(t,x,v)=\big(\psi(t,x),\partial_x\psi(t,x)[v]\big)$.
	Hence,  
	\[
		K^c(L)\big(\psi^c(t,x,v)\big)=
		\dfrac{\partial \big (L\circ \psi^c\big)}{\partial t}(t,x,v)
	\]
	and, since 
	\begin{equation}
		\label{eq:L-Kinvariant}
		\dfrac{\partial \big (L\circ \psi^c\big)}{\partial t}(t,x,v)=0,
	\end{equation}
	we get 
	\begin{equation}
		\label{eq:Killing-LocCoord}
		K^c(L)(x,v) 
		= K^h(x)\frac{\partial L }{\partial x^h}(x,v)
		+ \frac{\partial K^h}{\partial x^i}(x)v^i \frac{\partial L}{\partial v^h}(x,v)
		=0.
	\end{equation}
\end{remark}

\begin{remark}
	Since $K$ is an infinitesimal symmetry of $L$,
	by Noether's theorem, the Noether charge 
	is constant for any weak solution $z$ of the Euler-Lagrange equation
	of the Lagrangian $L$,
	independently from the boundary conditions. 
	This can be seen by recalling that a weak solution $z=z(s)$
	of the Euler-Lagrange equation is a $C^1$ curve
	(for fixed end points boundary conditions, see Appendix~\ref{sec:regularity})
	that locally  (i.e. in natural local coordinates of $TM$)
	satisfies the system of equations
	\begin{equation}
		\label{eq:Euler-Lagrange}
		\frac{\partial L}{\partial x^i}\big(z(s),\dot{z}(s)\big)
		= \frac{\de}{\de s}
		\left(\frac{\partial L}{\partial v^i}\big(z(s),\dot{z}(s)\big)\right),
		\quad \forall i = 0,\dots,m,
	\end{equation}
	hence from \eqref{eq:Killing-LocCoord} we get
	\begin{align*}
		\lefteqn{\frac{\de  }{\de s}\left(\frac{\partial L}{\partial v^i}\big(z(s),\dot{z}(s)\big)K^i(z(s))\right)}&\\
																														 &\quad=\frac{\de }{\de s}\left(\frac{\partial L}{\partial v^i}\big(z(s),\dot{z}(s)\big)\right)K^i(z(s))+\frac{\partial L}{\partial v^i}\big(z(s),\dot{z}(s)\big)\frac{\partial K^i}{\partial x^h}(z(s))\dot z^h(s)\\
																																			 &\quad=\frac{\partial L}{\partial x^i}\big(z(s),\dot{z}(s)\big)K^i(z(s))+\frac{\partial L}{\partial v^i}\big(z(s),\dot{z}(s)\big)\frac{\partial K^i}{\partial x^h}(z(s))\dot z^h(s)=0.
	\end{align*}
\end{remark}

\bigskip
Let us introduce a Lagrangian $L_c$ on $TM$ defined as 
\begin{equation}
	\label{L}
	L_c(x,v):= L(x,v)- \frac{Q^2(v)}{Q(K)}.
\end{equation}

\begin{proposition}
	\label{prop:Lproperties}
	The  following statements hold:
	\begin{list1}
	\item\label{OKL1}
		$L_c\in C^1(TM)$;

		\medskip
		\noindent \hspace{-1.4cm} and, for all $(x,v)\in TM$:
		\medskip

	\item\label{rewriteLc}
		\begin{equation}
			\label{QK}
			Q_x(K)  = 2\left( L(x,K) - L(x,0)-d(x) \right);
		\end{equation}
	\item\label{menoQ}
		\begin{align}
			&L_c(x,0)=L(x,0),\nonumber \\ 
			&L(x,K) + L_c(x,K) = 2 \big(L(x,0)+d(x)\big);\nonumber\\
			\intertext{and}
			&\partial_v L_c(x,v)[K]=-Q(v)+d(x),\label{Nc}
	\end{align}

	\item\label{LflowK}
		the flow of $K$ preserves also $L_c$, i.e. $K^c(L_c)=0$.
	\end{list1}
\end{proposition}

\begin{proof}
	Statement \eqref{OKL1} comes immediately from \eqref{L}
	and Assumption~\ref{ass:L}-\eqref{c1}.
	Let us prove \eqref{rewriteLc}. 	Let $x \in M$ be a given point and 
	let $l\colon \mathbb{R} \to \mathbb{R}$ be defined as
	\[
		l(\alpha) = L(x,\alpha K) - L(x,0).
	\]
	Hence, $l(0) = 0$ and, by Assumption~\ref{ass:L}-\eqref{linear},
	we obtain
	\[
		l'(\alpha) = \partial_v L(x,\alpha K)[K] = \alpha Q_x(K) + d(x).
	\]
	As a consequence, the function $l$ is equal to
	\[
		l(\alpha) = \frac{\alpha^2}{2}Q_x(K)+\alpha d(x).
	\]
	Therefore, noticing that $Q_x(K) = 2\big( l(1)-d(x))$,
	we obtain \eqref{QK}.
	Now	\eqref{menoQ} is a simple consequence of \eqref{L} and \eqref{QK}.

	Let us prove \eqref{LflowK}. 
	From \eqref{L} it is enough to prove that $Q$ and $Q(K)$ are invariant by
	the flow of $K^c$ and $K$, respectively.
	Let us consider $Q$ as a function on $TM$, i.e. $Q(x,v):=Q(v)$,
	thus we have to show that $K^c(Q)=0$.
	By \eqref{negativeQK}, $K_x\neq 0$ for all $x\in M$,
	thus for each $\bar x\in M$ we can take a neighborhood $U$ of $\bar x$
	and a coordinate system $(x^0, x^1, \ldots, x^m)$ defined in $U$
	such that $\frac{\partial}{\partial x^0}=K|_U$.
	Therefore, in such a coordinate system,
	\[
		Q(x, v)=\frac{\partial L}{\partial v^h}(x,v)K^h-d(x)
		=\frac{\partial L}{\partial v^0}(x,v)-d(x).
	\]
	Since $Q$ and $d$ are $C^1$, we know that 
	$\frac{\partial L}{\partial v^0}$ admits continuous partial derivatives
	w.r.t. the coordinates $(x^0,x^1,\ldots, x^m, v^0,v^1,\ldots,v^m)$ in $TU$.
	Notice also that, from \eqref{eq:Killing-LocCoord}, $K^c(L)=0$ is
	equivalent to $\frac{\partial L}{\partial x^0}(x,v)=0$.
	Being then a constant function, $\frac{\partial L}{\partial x^0}$ admits
	zero partial derivatives w.r.t. the coordinates $(x^0,x^1,\ldots, x^m,
	v^0,v^1,\ldots,v^m)$ as well.
	As $d$ is invariant by the flow of $K$, we have $\frac{\partial d}{\partial x^0}=0$ on $U$. Thus, from \eqref{eq:lift-LocCoor},
	we then get
	\[
		K^c(Q)(x,v)=\frac{\partial^2 L}{\partial x^0\partial v^0}(x,v)
		=\frac{\partial^2 L}{\partial v^0\partial x^0}(x,v)=0.
	\] 
	Since $Q(x,K)=\frac{\partial L}{\partial v^0}\big(x,(1,0,\ldots,0)\big)$,
	we also have
	\[
		K^c\big(Q(x,K)\big)=K\big(Q(x,K)\big)
		=\frac{\partial^2L}{\partial x^0\partial v^0}\big(x,(1,0,\ldots,0)\big)=0.
	\] 
\end{proof}
\begin{remark}
From \eqref{Nc} and \eqref{LflowK} in Proposition~\ref{prop:Lproperties},
we have that, like $L$, $L_c$ has affine Noether charge as well.
\end{remark}

Recalling Remark~\ref{atlas},
the following assumption ensures some 
growth conditions on $L_c$, often
used in critical point theory for the action functional of a Lagrangian
(see, e.g., \cite{Abbondandolo2007,Benci86}),
and its pointwise strong convexity.
\begin{assumption}\label{ass:L1}
	The Lagrangian $L_c\colon TM \to \mathbb{R}$,
	defined as in \eqref{L},
	satisfies the following assumptions:
	\begin{list1}
	\item
		there exists a continuous function
		$C\colon M\to(0,+\infty)$ 
		such that for all $(x,v)\in TM$,
		the following inequalities hold:
		\begin{align}
			\label{pinched}
			L_c(x,v)
				&\leq C(x)\big(\norm{v}^2+1\big);\\
				\label{partialxpinched}
			\|\partial_x L_c(x,v)\| &\leq C(x)\big(\norm{v}^2+1\big);	\\
			\label{partialvpinched}
			\|\partial_v L_c(x,v)\| &\leq C(x)\big(\norm{v}+1\big); 
		\end{align}
	\item		\label{monotonehyp}
		there exists a continuous function $\lambda\colon M \to (0,+\infty)$ such that 
		for each $x\in M$ and for all $v_1, v_2\in T_xM$,
			the following inequality holds:
		\begin{equation}
			\label{monotone}
			\big(\partial_vL_c(x,v_2)-\partial_vL_c(x,v_1)\big)[v_2-v_1]
			\geq \lambda(x)\norm{v_2 - v_1}^2;
		\end{equation}
	\end{list1}
\end{assumption}

\begin{remark}
	\label{rem:quadgrowth}
	We notice that from \eqref{monotone}
	and \eqref{Nc}
	we obtain 
	\begin{align*}
		Q_x(K) &= Q_x(K)- Q_x(0) \\
		&= \big(\partial_vL_c(x,0) - \partial_vL_c(x,K)\big)[K]
		\leq -\lambda (x)\norm{K}^2.
	\end{align*}
	Moreover, for all $(x,v)\in TM$ we have 
	\begin{multline*}
		L_c(x,v)-L_c(x,0)=
		\inte\frac{\de}{\de s}L_c(x,sv)\de s=
		\inte\partial_vL_c(x,sv)[v]\de s\\ =
		\inte\frac{1}{s}\bigg(\partial_vL_c(x,sv)[sv]-\partial_vL_c(x,0)[sv]\bigg)\de s
		+\partial_vL_c(x,0)[v] \\
		\ge \frac{1}{2}\lambda(x)\|v\|^2-\|\partial_vL_c(x,0)\|\|v\|.
	\end{multline*}
	Thus, $L_c$ satisfies the growth condition
	\[
		L_c(x,v)\geq L_c(x,0)-
		\frac{1}{\lambda(x)}\|\partial_vL_c(x,0)\|^2+\frac{\lambda(x)}{4}\norm{v}^2,
	\]
	and since $L_c(x,0)=L(x,0)$ and $\partial_vL_c(x,0)=\partial_vL(x,0)$ we get 
	\begin{equation}
		\label{quadgrowth}
		L_c(x,v)\geq L(x,0)
		-\frac{1}{\lambda(x)}\|\partial_vL(x,0)\|^2+\frac{\lambda(x)}{4}\norm{v}^2.
	\end{equation}

\end{remark}

The next and final assumption is needed to get a compactness condition 
on the sublevels of the reduced action functional
(see Lemma~\ref{lem:zH1bounded-Jc})
and then in the proof of the Palais-Smale condition for the same functional.
\begin{assumption} 
	\label{ass:bounds}
	There exist four  constants, $c_1,c_2,c_3, k_1, k_2$
	such that, for all $x \in M$, the following inequalities hold:
	\begin{align}
			&0<c_1\leq 
			\lambda (x),\label{eq:boundC1}\\
			\label{eq:boundC2}
			&L(x,0)\geq c_2\quad \text{and}\quad  \|\partial_vL(x,0)\|\leq c_3,\\
			\label{eq:boundNoether}
			&0<k_1\leq -Q_x(K),\\
			\label{dbounded}
			&|d(x)|\leq k_2.
	\end{align}
\end{assumption}

\section{Some classes of examples}
\label{sec:examples}
In this section we present various type of Lagrangians that satisfy
Assumptions~\ref{ass:L}--\ref{ass:bounds}.
We start with a generalization of the
Lorentz-Finsler Lagrangians studied in \cite{caponio2018}.
\begin{example}\label{ex:general-ind-Lagrangian}
	Let $S$ be a smooth $m$--dimensional manifold and $M = S \times \mathbb{R}$. 
	Let $g_S$ be a complete auxiliary Riemannian metric on $S$,
	whose associated norm is denoted by $\norm{\cdot}_S$,
	and let $g$ be the product metric $g=g_S\oplus dt^2$.
	Let 
	$L\colon TM \to \mathbb{R}$ be a Lagrangian on $M$ defined as 
	\begin{equation}
		\label{eq:gen-L}
		L\big((x,t),(\nu,\tau)\big)=
		L_{0}(x,\nu) + 2\big(\omega(\nu)+d(x)/2\big)\tau - \beta(x)\tau^2,
	\end{equation}
	where
	\begin{list1}
	\item\label{L0} $L_0\colon TS \to \mathbb{R}$ belongs to 
		$C^1(TS)$
		and 
		there exists a  continuous positive function
		$\ell: S \to (0,+\infty)$ such that
		\begin{align}
			\label{eq:genL-pinched}
			L_0(x,\nu)
				&\leq \ell(x)\big(\norm{\nu}_S^2+1\big);\\
				\label{eq:genL-partialxpinched}
			\|\partial_x L_0(x,\nu)\|_S &\leq \ell(x)\big(\norm{\nu}_S^2+1\big);	\\
			\label{eq:genL-partialvpinched}
			\|\partial_\nu L_0(x,\nu)\|_S &\leq \ell(x)\big(\norm{\nu}_S+1\big);
		\end{align}
	\item\label{L0convex} $L_0$ is pointwise strongly convex,
		i.e. there exists a continuous function
		$\lambda_0\colon S\to (0,+\infty)$
		such that, for all $x\in S$ and all $\nu_1,\nu_2\in T_xS$,
		\eqref{monotone} holds
		with $L_c$ replaced by $L_0$, $\lambda$ by $\lambda_0$ and
		$\|\cdot\|$ by $\|\cdot\|_S$;
	\item\label{omegabeta} $\omega$ is a $C^1$ one-form on $S$,
		$d\colon S\to \mathbb{R}$ is a $C^1$ function and
		$\beta \colon S \to (0,+\infty)$ is a $C^1$ positive function.
	\end{list1}	
	In this case, the field $K=\partial_t\equiv (0,1)$
	is an infinitesimal symmetry of $L$ and	
$d$ is invariant by the flow of $K$, because it is a function on $S$.
	Notice that
	if $L_0$ is the square of a Riemannian norm on $S$ and $d=0$
	then $L$ is the quadratic form associated with the
	Lorentzian metric of a standard stationary spacetime
	(see, e.g., \cite{GiannoniMasiello1991}).
	Moreover, if $L_0$ is the square of the norm of a Riemannian metric
	plus a one-form	
	$\omega_0$
 on $S$,
	then they include electromagnetic type Lagrangians
	on a standard stationary Lorentzian manifold
	with an exact electromagnetic field on $S\times\R$
	having a potential one-form $\omega_0\oplus d(x)\de t$,
	see Remark~\ref{FV} below.
\end{example}

In the next result we show that $L$ defined as in \eqref{eq:gen-L}
satisfies Assumptions~\ref{ass:L}-\ref{ass:L1}
and we give some further conditions ensuring that
it also satisfies Assumption~\ref{ass:bounds}.

\begin{proposition}
	A Lagrangian $L$ defined as in \eqref{eq:gen-L},
	such that \eqref{L0}--\eqref{omegabeta} above hold, satisfies
	Assumptions \ref{ass:L} and \ref{ass:L1}. 
	Moreover, if there exist some constants
	$b, \ell_1, \ell_2, \ell_3, \ell_4$,
	such that
	$\beta(x) \ge b>0$, $\lambda_0(x)\geq  \ell_1>0$,
	$L_0(x,0)\geq \ell_2$,
	$\|\partial_vL_0 (x,0)\|\leq \ell_3$ and $|d(x)|\leq \ell_4$, for every $x\in S$,
	then $L$ satisfies Assumption \ref{ass:bounds}.
\end{proposition}
\begin{proof}
	As remarked above,
	the vector field $\partial_t\equiv(0,1)$ is an infinitesimal symmetry for $L$;
	moreover, since by hypotheses $L_0$, $\omega$ and $\beta$ are of class $C^1$,
	$L \in C^1(TM)$ as well.
	A direct computation shows that
	\begin{equation}
		\label{noetherxample} 
		\partial_vL\big((x,t),\cdot\big)[(0,1)]
		= 2\big(\omega_x- \beta(x)\diff t\big)+ d(x)
	\end{equation}
	which is an affine function  on $TM$ that we denote by $N$.
	Let $Q:=2(\omega-\beta\de t)$, hence
	\begin{equation}
		\label{eq:estimateQ-mainexample}
		Q(K) = Q\big((0,1)\big)
		= -2\beta < 0,
	\end{equation}
	and thus the conditions in Assumption~\ref{ass:L} are satisfied.
	Using \eqref{L} with \eqref{eq:estimateQ-mainexample},
	we see that $L_c\colon TM \to \mathbb{R}$ 
	is given by
	\begin{multline}
		\label{eq:genL1}
		L_c((x,t),(\nu,\tau))=
		L_0(x,\nu) +
		\left( \frac{1}{\sqrt{\beta(x)}}\omega(\nu)- \sqrt{\beta(x)}\tau \right)^2\\
		+\frac{1}{\beta(x)}\omega^2(\nu) +\frac{d(x)}{2}\tau.
	\end{multline}
	Let us show that $L_c$ satisfies
	\eqref{pinched}, \eqref{partialxpinched} and \eqref{partialvpinched}.
		By \eqref{eq:genL-pinched} we have 
	\begin{multline*}
		L_c((x,t),(\nu,\tau)) \le
		\ell(x)(\norm{\nu}^2_S + 1) 
		\\+ \frac{2\omega^2(\nu)}{\beta(x)}
		+ 2\beta(x)\tau^2
		+ \frac{\omega^2(\nu)}{\beta(x)}+\frac{d^2(x)}{2}+\frac{\tau^2}{2},
	\end{multline*}
	so setting
	\[
		C((x,t))\equiv C(x) = \max\left\{
			\ell(x),
			\frac{3\norm{\omega_x}_S^2}{\beta(x)},2\beta(x)+\frac{1}{2},
			\frac{d^2(x)}{2}
		\right\}
	\]
	\eqref{pinched} holds.
	Let us compute  $\partial_{(x,t)}L_c$:
	\begin{multline*}
		\partial_{(x,t)} L_c\big((x,t),(\nu,\tau)\big)[\xi,\zeta]
		= 
		\partial_x L_0(x,\nu)[\xi]
		- 2\partial_x\omega(\xi,\nu)\tau
		+ \de\beta(\xi)\tau^2\\
		+ \frac{4}{\beta(x)}\omega(\nu)\partial_x\omega(\xi, \nu)
		- \frac{2}{\beta^2(x)}\omega^2(\nu)\de\beta(\xi)+\de d(\xi)\tau,
	\end{multline*}
	Hence, 
	\begin{multline*}
		\norm{\partial_{(x,t)} L_c\big((x,t),(\nu,\tau)\big)}
		\le \norm{\partial_x L_0(x,\nu)}
		+ 2\norm{(\partial_x\omega)_x}_S\norm{\nu}_S\abs{\tau}\\
		+ \norm{(\de\beta)_x}_S\abs{\tau}^2
		+ \frac{4}{\beta(x)}\norm{(\partial_x\omega)_x}_S\norm{\omega_x}_S\norm{\nu}_S^2\\
		+ \frac{2}{\beta^2(x)}\norm{\omega_x}_S^2\norm{(\de \beta)_x}_S\norm{\nu}_S^2+\|(\de d)_x\||\tau|.
	\end{multline*}
	By \eqref{eq:genL-partialxpinched} and recalling that 
	$\norm{(\nu,\tau)}^2 = \norm{\nu}_S^2 + \abs{\tau}^2$,
	we infer the existence of a function $C\colon M \to (0,+\infty)$
	such that \eqref{partialxpinched} holds.
	Similarly, using \eqref{eq:genL-partialvpinched}
	we obtain \eqref{partialvpinched}.

	Let us show that $L_c$ satisfies \eqref{monotone}.
	From \eqref{eq:genL1} we have
	\begin{multline*}
		\partial_{(\nu,\tau)} L_c\big((x,t),(\nu,\tau)\big)[(\nu_1,\tau_1)]
		=\partial_vL_0(x,\nu)[\nu_1]\\
		+2	\left( \frac{1}{\sqrt{\beta(x)}}\omega(\nu)- \sqrt{\beta(x)}\tau \right)
		\left( \frac{1}{\sqrt{\beta(x)}}\omega(\nu_1)- \sqrt{\beta(x)}\tau_1 \right)\\
		+\frac{2}{\beta(x)}\omega(\nu)\omega(\nu_1)+d(x)\tau_1,
	\end{multline*}
	hence using that $L_0$ is pointwise strongly convex we get 
	\begin{multline*}
		\Big(\partial_{(\nu,\tau)}L_c\big((x,t),(\nu_2,\tau_2)\big)-\partial_{(\nu,\tau)}L_c\big((x,t),(\nu_1,\tau_1)\big)\Big)[(\nu_2-\nu_1,\tau_2-\tau_1)]\\
		\geq \lambda_0(x)\|\nu_2-\nu_1\|^2_S+\frac{4}{\beta(x)}\omega_x^2(\nu_2-\nu_1)\\
		\phantom{some space}
		+2\beta(x)(\tau_2-\tau_1)^2
		-4(\tau_2-\tau_1)\omega_x(\nu_2-\nu_1)\\
		\geq \lambda_0(x)\|\nu_2-\nu_1\|^2_S+\beta(x)(\tau_2-\tau_1)^2,
	\end{multline*}
	thus \eqref{monotone} holds by taking 
	\begin{equation}
		\label{lambdax}
		\lambda(x):=\min\{\lambda_0(x),\beta(x)\}.
	\end{equation}

	It remains to prove Assumption \ref{ass:bounds}.
	Of course, \eqref{dbounded} is trivially satisfied and
	if $\beta\ge b > 0$
	then by \eqref{eq:estimateQ-mainexample} we obtain
	\eqref{eq:boundNoether}.
	By \eqref{lambdax},
	we also have \eqref{eq:boundC1} with $c_1=\min\{\ell_1,b\}$.
	As $L\big((x,t), 0\big)=L_0(x,0)$ and
	$\partial_{(\nu,\tau)}L\big((x,t),0\big)=\partial_\nu L_0(x,0)$,
	\eqref{eq:boundC2} is satisfied as well with $c_2=\ell_2$ and $c_3=\ell_3$.
\end{proof}

\begin{remark}\label{FV}
	A special case of a Lagrangian 
	in Example~\ref{ex:general-ind-Lagrangian} 
	that satisfies our assumptions
	is given by \eqref{eq:gen-L} with
	\[
		L_0(x,\nu) = F^{2}(x,\nu) +\omega_0(\nu) +V(x),
	\]
	where 
	$V\colon S \to \mathbb{R}$ is a $C^1$ function bounded from below,
	 $\omega_0$ is a $C^1$ one-form on $S$,
	such that $\sup_{x\in S}\|(\omega_0)_x\|_S<+\infty$, and 		
	$F\colon TS \to [0,+\infty)$ is a $C^1$ Finsler metric on $S$,
	i.e it is a non-negative,
	$C^1$ Lagrangian on $TS$, positively homogeneous of degree $1$ w.r.t. $\nu$,
	such that $F^2$ is pointwise strongly convex,
	i.e. it satisfies \eqref{monotone} on $TS$.
	We remark that usually in the definition of a Finsler metric it is assumed that
	$F^2\in C^2(TS\setminus 0)$
	(where $0$ denotes the zero section of $TS$) and its vertical Hessian,
	the so-called fundamental tensor $g_F$,
	\[		
		g_F(x,v)[u,w]:=
		\frac 1 2\frac{\partial^2 F^2}{\partial s \partial t}(x, v + tu + sw)
		\bigg|_{(s,t) = (0,0)}
	\]
	for all $(x,v)\in TM\setminus 0$ and all $u,w \in T_xM$,
	is assumed to be positively homogeneous of degree $0$ in $v$
	and  positive definite for all $(x,v)\in TM\setminus 0$ (see, e.g., \cite{BaChSh00}).
	Inequality \eqref{monotone} for $F^2$ on $TS$
	follows by the mean value theorem applied to the function
	$\nu\in TS\mapsto \partial_vF^2(x,v)[\nu_2-\nu_1]$, when $\nu_1, \nu_2$ are
	not collinear vectors with opposite directions or when one of them is $0$; 
	for collinear vectors with opposite directions it follows  by a continuity
	argument.
	Notice that  $\lambda_0(x)$ in \eqref{monotone} for $F^2$ is then
	equal to 
	\[
		\lambda_0(x)=2\min_{\nu\in T_x S\setminus\{0\}}
		\left(\min_{u\in T_xS\setminus\{0\}} g_F(x,\nu)\Big[\frac{u}{\|u\|_S},
		\frac{u}{\|u\|_S}\Big]\right),
	\]
	and that	 
	\eqref{eq:genL-pinched}--\eqref{eq:genL-partialvpinched}
	are ensured by	the  homogeneity of degree $2$ of $F^2$ w.r.t. $\nu$.
\end{remark}

We notice  that Lagrangians $L$ satisfying
Assumptions~\ref{ass:L}--\ref{ass:L1} are generated by Lagrangians $L_b$
satisfying \eqref{pinched}---\eqref{monotone} and
admitting a vector field $K$ as an infinitesimal symmetry
with affine  Noether charge $N_b=Q_b+d$ such that $Q_b(K)>0$.
Indeed, arguing as in Proposition~\ref{prop:Lproperties}-\eqref{LflowK},
$L:=L_b-\frac{Q_b^2}{Q_b(K)}$ admits $K$ as infinitesimal symmetry and its
Noether charge is 
\[
	N=N_b-2Q_b=-Q_b+d,
\]
hence $L_c=L_b$ and
then, of course, $L_c$ satisfies Assumption~\ref{ass:L1}.
This observation gives
rise to the following families of examples.

\begin{example}
	Let $M$ be a smooth $(m+1)$-dimensional manifold endowed with a complete
	Riemannian metric $g$ and
	$F$, $\omega_0$, $V$
	be respectively a Finsler metric, a one-form and a function on $M$,
	all of  $C^1$ class and invariant by the flow of a		nowhere vanishing
		vector field $K$ on $M$.
	Let us assume that the Noether charge associated with $F^2$ and $K$
	is a one-form of class $C^1$ on $M$.
		Let $L:TM\to\R$ be given by 
	\begin{equation}
		\label{FVQ}
		L=F^2+ \omega_0 +V-\frac{Q_F^2}{Q_F(K)},
	\end{equation}
	and $\lambda(x)$ be the positive continuous function in \eqref{monotone} for $F^2$ on $TM$.
\end{example}
\begin{proposition}
	Assume that there exist two constants $c_1>0$ and $a\geq 0$
	such that $\lambda(x)>c_1$, and $\|(\omega_0)_x\|\leq a$,
	for all $x\in M$,
	$V$ is bounded from below,
	$\inf_{x\in M}\|K_x\|>0$
	and $\sup_{x\in M}\|K_x\|<+\infty$. 
	Then $L$ in \eqref{FVQ} satisfies Assumptions~\ref{ass:L}--\ref{ass:bounds}.	
\end{proposition}
\begin{proof}
	Set $L_b=F^2+\omega_0+V$,
	then 	$L_b$ admits $K$ as an infinitesimal symmetry with affine Noether charge 
	\[
		N_b=N_F+\omega_0(K).
	\]
	Hence, the same holds for $L$ and the one-form 
	appearing in its Noether charge is $Q=-Q_F$.
	Since $\omega_0$ is invariant by the flow of $K$,
	the Lie derivative $\mathcal L_K\omega_0$ vanishes.
	In particular, $0=\mathcal L_K\omega_0(K)=K\big(\omega_0(K)\big)$,
	i.e. $\omega_0(K)$ is invariant by the flow of $K$.
	We also notice that $(Q_F)_x(K)=2F^2(x,K)>0$ for all $x\in M$ since,
	by assumption, $K_x\neq 0$ for all $x\in M$.
	Thus $L$ satisfies Assumption~\ref{ass:L}.
	As $L_c=L_b$, it  satisfies \eqref{pinched}--\eqref{partialvpinched}
	because  $F^2$ is positively homogeneous of degree two; moreover, it satisfies
	also \eqref{monotone} as $F^2$ is pointwise strongly convex.
	Since $L(x,0)=V(x)$, and $\partial_vL(x,0)=\omega_0$,
	\eqref{eq:boundC2} holds; being 
	\[
		-Q_x(K)=(Q_F)_x(K)= 2F^2(x, K)
		\ge c_1\|K\|_x^2\geq c_1\inf_{x\in M}\|K\|^2_x>0,
	\] 
	and $d=\omega_0(K)$,
	\eqref{eq:boundNoether} and \eqref{dbounded} hold as well.
\end{proof}

The next class of examples involves Lorentz-Finsler metrics $L_F$ as defined by
J. K. Beem in \cite{Beem70}
(see also \cite{GaPiVi12,Minguz15,Perlic06}). 

\begin{example}\label{beemxample}
	Let $M$ be a smooth manifold of dimension $m+1$, and $g$ an auxiliary complete Riemannian metric on $M$. Let   $L_F:TM\to\mathbb R$ be a Lagrangian which satisfies the following conditions: 
	\begin{list1}
	\item\label{LF-1}
		$L_F \in C^1(TM)\cap C^{2}(TM \setminus 0)$, where $0$ denotes the zero section of $TM$;
	\item
		$L_F(x,\lambda v)=\lambda^2 L_F(x,v)$ for all $v \in TM$ and all $\lambda>0$;
	\item \label{LF-3}
		for any $(x,v) \in TM \setminus 0$,
		the vertical Hessian of $L_F$, i.e. the symmetric matrix
		\[
			(g_F)_{\alpha \beta}(x,v)
			:= \frac{\partial^2 L_F}{\partial v^\alpha \partial v^\beta}(x,v),
			\quad \alpha,\beta=0,1,\ldots,m,
		\]
		is non-degenerate with index $1$.	
	\end{list1}
\end{example}
Let us assume that $L_F$ admits a nowhere vanishing vector field $K$
as an infinitesimal symmetry and that its Noether charge
is equal to
$N_{L_F}=Q_{L_F}$, where $Q_{L_F}$ 
is a $C^1$ one-form such that $Q_{L_F}(K) < 0$.
Let $L=L_F+\omega_1+V$ where $\omega_1$ and $V$ are, respectively,
a $C^1$ one-form on $M$, such that
$\sup_{x\in M}\|(\omega_1)_x\|<+\infty$,
and a $C^1$ function, bounded from below   on $M$.
We assume that both $\omega_1$ and $V$ are
invariant by the flow of $K$.
Then the Noether charge of $L$ is $N=N_{L_F}+\omega_1(K)=Q_{L_F}+\omega_1(K)$.
Thus, $L$ satisfies Assumption~\ref{ass:L}.
Let $Q:=Q_{L_F}$; so $L_c$ is  equal to $L_c= L-\frac{Q^2}{Q(K)}$. 
\begin{proposition}
	If conditions \eqref{LF-1}--\eqref{LF-3}
	of Example \ref{beemxample} hold,
	then $L_c$ satisfies Assumption~\ref{ass:L1}. 
\end{proposition}
\begin{proof}
	Let us show that	$L_c$ admits vertical Hessian 
	at any $(x,v)\in TM\setminus 0$, which is a positive definite bilinear form on $T_xM$.
	We observe that for any $(x,v)\in TM\setminus 0$, we have
	\begin{align}
	\label{fiberHessL1}
	\partial_{vv}L_c(x,v)
			&=\partial_{vv}L(x,v)-\frac{2}{Q(K)}Q\otimes Q\nonumber\\
			&=\partial_{vv}L_F(x,v)-\frac{2}{Q(K)}Q\otimes Q
	\end{align}
	As for each $u\in T_x M$ we have
	\begin{multline}
		\label{orto}
		\partial_{vv}L(x,v)[K,u]=\frac{\partial^2 L}{\partial s \partial t}
		(x, v + tK + su)\bigg|_{(s,t) = (0,0)}\\
		=\frac{\partial \big(\partial_v L(x,v+su)[K]\big)}{\partial s }\bigg|_{s= 0}
		=\frac{\partial Q(v+su)}{\partial s }\bigg|_{s= 0}=Q(u),
	\end{multline}
	from \eqref{negativeQK} and \eqref{fiberHessL1},
	we get $\partial_{vv}L_c(x,v)[K,K]=-Q(K)>0$.
	Let $w\in \ker Q$.
	From \eqref{orto}, we have
	$\partial_{vv}L(x,v)[w,K]=0$,
	and since $\partial_{vv}L(x,v)$ has index $1$,
	we also have 
	\[\partial_{vv}L_c(x,v)[w,w]=\partial_{vv}L(x,v)[w,w]>0,\]
	hence $\partial_{vv}L_c(x,v)[\cdot,\cdot]$ 
	is positive definite.
	Reasoning as in the last part of Remark~\ref{FV},
	we deduce that \eqref{monotone} holds. 
	From \eqref{fiberHessL1},  since $\partial_{vv}L_F(x,v)-\frac{2}{Q(K)}Q\otimes Q$ is continuous on $TM\setminus 0$ and positively homogeneous of degree $0$ in $v$,  we deduce \eqref{pinched} 
	and \[C(x)=\max\big\{\Lambda(x)+1, V(x)+\|(\omega_1)_x\|^2\big\},\] where 
	\[\Lambda(x):=\max_{\substack{v\in T_xM,\|v\|=1\\ w\in T_xM}} \left(\frac 1 2\partial_{vv}L_F-\frac{Q\otimes Q}{Q(K)}\right)(x,v)\left[\frac{w}{\|w\|}, \frac{w}{\|w\|}\right].
	\]
	Up to redefine $C(x)$, \eqref{partialxpinched} and \eqref{partialvpinched} can be obtained analogously.
\end{proof}
In particular, Lagrangians in Example~\ref{beemxample} include the class of $C^2$ stationary Lorentzian metrics. We also want to consider the $C^1$ case. 
\begin{example}\label{statlor}
	Let $(M,g_L)$ be a Lorentzian manifold of dimension $m+1$ with $C^1$ metric
	tensor $g_L$. Let $K$ be a timelike Killing vector field for $g_L$, i.e. $K$
	is a Killing vector field such that $g_L(K,K)<0$. Then $(M,g_L)$ is called a
	{\em stationary Lorentzian manifold}. 
	Let $L(x,v):=g_L(v,v)$; we notice that $L\in C^1(TM)$ and
	$\partial_vL(x,v)[K]=(g_L)_x(\cdot,K)$, thus $Q(K)=2g_L(K,K)<0$.
	The Lagrangian $L_c$ is equal to
	\[L_c(x,v)= g_L(v,v)- \frac{2 g_L(K,v)^2}{g_L(K,K)}\]
	and then it is equal to the square of the norm of a Riemannian metric
	$g_R$ (as in \cite{giannoni1999}).
	Thus, Assumption~\ref{ass:L1} is satisfied as well (by using the same metric
	$g_R$ as auxiliary Riemannian metric $g$), provided that $g_R$ is complete
	with
	\[
		C(x)=\max\left\{2,
			(m+1)\max_{k\in\{0,\ldots,m\}}
			\Big(\max_{v\in T_x M\neq 0}
				\frac{\partial (g_R)_{ij}}{\partial x^k}(x)\frac{v^i}{\|v\|}\frac{v^j}{\|v\|}
			\Big)
		\right\}.
	\] 
	Finally, Assumption~\ref{ass:bounds} is satisfied provided that
	there exists a constant $k_1$ such that $-g_L(K,K)\geq k_1>0$. 
\end{example}

	The following example of Lagrangians are 
	the Lorentz-Finsler Lagrangians studied in \cite{JavSan20}
	and they can be included in the class of Example~\ref{beemxample}
	(see Proposition \ref{prop:JavSan-in-Beem}).

\begin{example}\label{JS}
	Let $M$ be a smooth manifold and 
	\begin{equation}
		\label{LFJS}
		L_F=F^2-\omega^2,
	\end{equation}
	where $F$ and $\omega$ are, respectively, a Finsler metric of class $C^1(TM)\cap C^2(TM\setminus 0)$ and a one-form of class $C^1$ both invariant 
	invariant by the flow of a nowhere vanishing vector
	field $K$ and such that		the Noether charge associated with $F^2$ and $K$
    is a $C^1$  one-form on $M$,  $N_F = Q_F$.
		Then, $L_F$ admits $K$ as an infinitesimal
	symmetry and 
	$N_{L_F}=Q_F-2\omega(K)\omega=:Q_{L_F}$.
	Notice that $(Q_{L_F})_x(K)=2\big(F^2(x, K)-\omega^2(K)\big)$. 
\end{example}

\begin{proposition}
	\label{prop:JavSan-in-Beem}
	Assume that $Q_{L_F}(K)<0$, then $\omega_x(K)\neq 0$ for all $x\in M$
	and $L_F$ in \eqref{LFJS} is a Lagrangian
	of the type in Example~\ref{beemxample}.
\end{proposition}
\begin{proof}
	The non-trivial part of the statement is to prove \eqref{LF-3}
	in Example~\ref{beemxample}.
	For all $(x,v)\in TM\setminus 0$ we have:
	\begin{multline*}
		\partial_{vv}L_F(x,v)[K,K]
		= \partial_{vv}F^2(x,v)[K,K]-2\omega^2(K)\\
		= \partial_v\big(\partial_vF^2(x,v)[K]\big)[K] -2\omega^2(K)
		=\partial_v\big((Q_F)(v)\big)[K]-2\omega^2(K)\\
		=Q_F(K)-2\omega^2(K)=2\big(F^2(x, K)-\omega^2(K)\big)<0,
	\end{multline*}
	thus in particular we get that $\omega_x(K)\neq 0$, for all $x\in M$.
	Moreover for all $w\in \ker(\omega_x)$, $w \ne 0$,
	we have
	\begin{align*}
		\partial_{vv}L_F(x,v)[w,w]
		&=\partial_{vv}F^2(x,v)[w,w]-2\omega^2(w)\\
		&=\partial_{vv}F^2(x,v)[w,w]>0.
	\end{align*}
	Thus, being $K$ transversal to $\ker (\omega)$,
	we deduce that $\partial_{vv}L_F(x,v)$
	has index $1$ for all $(x,v)\in TM\setminus 0$.
\end{proof}

\section{The reduced manifold of paths and action}
\label{sec:problem-reduction}
Let $L:TM\to\R$ be a Lagrangian satisfying Assumptions~\ref{ass:L} and \ref{ass:L1}.
Recalling that $M$ is endowed with an auxiliary complete Riemannian metric $g$,
let us consider the set
\begin{multline*}
	W^{1,2}([0,1],M) :=
	\bigg\{
		z\colon [0,1] \to M:
		z \text{ is absolutely continuous and }\\
		\int_{0}^1 g(\dot{z},\dot{z})\de s < + \infty
	\bigg\},
\end{multline*}
and, for any two fixed points $p,q\in M$, its subset
\[
	\Omega_{p,q}^{1,2} :=
	\left\{z\in W^{1,2}([0,1],M): z(0) = p,\ z(1) = q
	\right\}.
\]
It is well known that since $(M,g)$ is complete, $W^{1,2}([0,1],M)$ is a smooth,
infinite dimensional, complete Riemannian manifold and
$\Omega_{p,q}^{1,2}$ is a smooth closed (hence complete) submanifold
(see, e.g., \cite[Lemma 6.2]{Ekelan74}).
For every $z \in \Omega_{p,q}^{1,2}$,
the tangent space $T_z\Omega^{1,2}_{p,q}$ is equal to 
\[
	T_z\Omega_{p,q}^{1,2} = 
	\left\{
		\xi \in W^{1,2}_0([0,1],TM):
		\xi(s) \in T_{z(s)}M,\ \forall s \in [0,1]
	\right\}.
\]
Weak solutions of 
\eqref{eq:Euler-Lagrange} connecting the points $p$, $q\in M$ are by definition 
the critical points of the action functional
$\mathcal A\colon \Omega_{p,q}^{1,2} \to \mathbb{R}$, defined as
\[
	\mathcal A(z): = \int_{0}^{1} L(z,\dot{z})\diff s.
\]
\begin{remark}
	From \eqref{L}, we have that $L=L_c+Q^2/Q(K)$ and hence from Assumption~\ref{ass:L1}
	we get that $\A$ is $C^1$ on $\Omega_{p,q}^{1,2}$
	(see, e.g., the first part of the proof of Proposition 3.1 in \cite{Abbondandolo2009}), with differential  
	$\diff\mathcal A(z)\colon T_z\Omega_{p,q}^{1,2} \to \mathbb{R}$ at a curve $z\in \Omega_{p,q}^{1,2}$ equal to
	\begin{equation}
		\label{eq:diff-f}
		\diff\mathcal A(z)[\xi]= \int_{0}^{1}
		\left(
			\partial_x L(z,\dot{z})[\xi]
			+ \partial_{v}L(z,\dot{z})[\dot{\xi}]
		\right) \diff s.
	\end{equation}
\end{remark}

Let $\xi \in T_z\Omega_{p,q}$ such that $\xi=X\circ z$,
with $X$ a smooth vector field  in $M$, then  in natural coordinates  
$(x^0,\dots,x^m,v^{0},\dots,v^m)$, of $TM$,
the integrand function in \eqref{eq:diff-f} is given by
\[
	\partial_x L(z,\dot{z})[\xi]
	+ \partial_{v}L(z,\dot{z})[\dot{\xi}] =
	\frac{\partial L}{\partial x^i}(z,\dot{z}) X^i(z)
	+ \frac{\partial L}{\partial v^i}(z,\dot{z})
	\frac{\partial X^i}{\partial x^h}(z)\dot z^h.
\]	
In the following, by an abuse of notation,
we also denote by $\dot X$ the derivative of $X(z)$,
i.e. $\dfrac{\partial X^i}{\partial x^h}(z)\dot z^h$.

From \eqref{eq:Killing-LocCoord}
we then get
\begin{equation}
	\label{eq:hp-K1}
	\partial_x L(z,\dot{z})[K]+
	\partial_{v}L(z,\dot{z})[\dot{K}] = 0,
\end{equation}
for all $z \in \Omega_{p,q}^{1,2}$.

The main goal of this section is to 
prove that the critical points of $\mathcal A$
lay on the following subset of $\Omega_{p,q}^{1,2}$:
\begin{equation}
	\label{eq:def-Npq}
	\mathcal{N}_{p,q}:=	\left\{
		z \in \Omega_{p,q}^{1,2}(M):
		N(z,\dot z)\text{ is constant a.e. on }[0,1]
	\right\}.
\end{equation}
For every $z \in \Omega_{p,q}^{1,2}$,
let us define
\bmln
\mathcal{W}_z := \left\{
	\xi \in T_z\Omega_{p,q}^{1,2}:
\exists \mu\in W^{1,2}_0([0,1],\R)\right.\\
\left.\text{ such that $\xi(s)=\mu(s)K_{z(s)}$},\text{ a.e. on } [0,1]
\right\}.
\emln
\begin{proposition}
	\label{prop:NpqCost}

	\[	\mathcal{N}_{p,q} =
		\left\{
			z \in \Omega_{p,q}^{1,2}:
			\diff\mathcal A(z)[\xi] = 0,\ \forall \xi \in \mathcal{W}_z
	\right\}.\]
\end{proposition} 
\begin{proof}
	For all $\xi\in \mathcal W_z$, from \eqref{eq:hp-K1} we have
	\begin{multline*}
		\diff\mathcal A(z)[\xi]= 
		\int_{0}^{1} \left(	
			\partial_x L(z,\dot{z})[\xi]
			+ \partial_{v}L(z,\dot{z})[\dot{\xi}]
		\right) \diff s \\
		= \int_{0}^{1} \mu\left(
			\partial_x L(z,\dot{z})[K]
			+ \partial_{v}L(z,\dot{z})[\dot{K}]
		\right) \diff s
		+\int_{0}^{1}\mu'\ \partial_{v} L(z,\dot{z})[K] \diff s \\
		=\int_{0}^{1}\mu'\ \partial_{v} L(z,\dot{z})[K] \diff s .
	\end{multline*}
	As a consequence, $\diff\mathcal A(z)[\xi] = 0$
	for all $\xi \in \mathcal{W}_z$ if and only if
	\[
		\int_{0}^{1}\mu'\ \partial_{v} L(z,\dot{z})[K] \diff s = 0,
		\quad \forall \mu \in W^{1,2}_0([0,1],\mathbb{R}),
	\]
	namely if and only if
	$
	\partial_v L(z,\dot{z})[K]=N(z,\dot z)
	$
	is constant a.e. on $[0,1]$. 
\end{proof}

\begin{proposition}
	\label{prop:Npqmanifold}
	The set $\mathcal{N}_{p,q}$
	is a $C^1$ closed submanifold of $\Omega_{p,q}^{1,2}$.
	Moreover, for every $z \in \mathcal{N}_{p,q}$,
	the tangent space $T_z\mathcal{N}_{p,q}$ is given by
	\begin{equation}
		\label{eq:TzNpq}
		T_z\mathcal{N}_{p,q} = 
		\left\{
			\xi \in T_z\Omega^{1,2}_{p,q}:
			\partial_{x}N(z,\dot{z})[\xi]+Q(\dot\xi)
			\text{ is constant a.e. on }[0,1]
		\right\}.
	\end{equation}
\end{proposition}
\begin{proof}
	Let 
	$F:\Omega_{p,q}^{1,2}\to L^{2}([0,1],\mathbb{R})$
	be defined as
	\[
		F(z):=N(z,\dot{z})
	\]
	and $\mathcal{C}\subset L^2([0,1],\mathbb{R})$
	be defined as
	\[
		\mathcal{C}:=		\left\{
			f \in L^{2}([0,1],\mathbb{R}):\ f(s)=\text{const. a.e.}
		\right\}.
	\]
	By the definition of $\mathcal{N}_{p,q}$ given in \eqref{eq:def-Npq},
	we have
	\[
		\mathcal{N}_{p,q} = F^{-1}(\mathcal{C}).
	\]
	The map $F$ is $C^1$ 
	and its differential is
	\begin{equation}
		\label{eq:Npqmanifold-diffF}
		\diff F(z)[\xi]=
		\partial_{x}N(z,\dot{z})[\xi]+Q(\dot\xi).
	\end{equation}
	By \cite[Proposition 3, p. 28]{lang1985},
	it is enough to show that for all $z \in \mathcal{N}_{p,q}$
	and $h \in L^2([0,1],\mathbb{R})$
	there exist $\xi \in T_z\Omega_{p,q}^{1,2}$ 
	and $c \in \mathbb{R}$ such that
	\begin{equation}
		\label{eq:Npqmanifold-proof1}
		\diff F(z)[\xi] = h + c.
	\end{equation}
	Therefore,
	let us fix $z \in \mathcal{N}_{p,q}$
	and $h \in L^2([0,1],\mathbb{R})$.
	Let us consider $\xi \in \mathcal{W}_z \subset T_z\Omega_{p,q}^{1,2}$,
	so there exists 
	$\mu \in W^{1,2}_0([0,1],\mathbb{R})$
	such that $\xi(s) = \mu(s)K(z(s))$.
	By \eqref{eq:Npqmanifold-diffF},
	recalling that $d$ is invariant by the flow of $K$ and then $\de d(K)=0$,
	we obtain
	\begin{equation}
		\label{eq:Npqmanifold-diffFxi}
		\diff F(z)[\xi]=
		\mu\big(\partial_{x}Q(\dot z, K)+Q(\dot K)\big)+\mu'Q(K(z)).
	\end{equation}
	Using \eqref{eq:Npqmanifold-diffFxi} and
	recalling that by Assumption \ref{ass:L}-\eqref{negativeQKhyp},
	$Q_x(K) \neq  0$ for all $x\in M$,
	\eqref{eq:Npqmanifold-proof1}
	becomes an ODE in normal form with respect to $\mu$,
	namely
	\begin{equation}
		\label{eq:Npqmanifold-ODE}
		\mu'(s) + a(s)\mu(s) = b_c(s),
	\end{equation}
	where
	\[
		a(s) = \frac{\partial_{x}Q(\dot z, K)+Q(\dot K)}{Q(K(z))}
		\quad \text{and}\quad
		b_c(s) = \frac{h(s) + c}{Q(K(z))}.	
	\]
	Setting
	$
	A(s)=\int_{0}^{s}a(\tau)\diff\tau,
	$
	and
	\[
		c = - \left(\int_{0}^{1}
			\frac{e^{A(s)}}{Q(K(z))}\diff s
		\right)^{-1}
		\left(\int_{0}^{1}
			\frac{e^{A(s)}h(s)}{Q(K(z))}\diff s
		\right),
	\]
	a solution of \eqref{eq:Npqmanifold-ODE}
	which satisfies the boundary conditions 
	$\mu(0) = \mu(1) = 0$ is given by
	\begin{equation*}
		\mu(s)=e^{-A(s)}\int_{0}^{s}b_c(s)e^{A(\tau)}\diff\tau.
	\end{equation*}
	Thus, for every $z \in \mathcal{N}_{p,q}$
	and $h\in L^{2}([0,1],\mathbb{R})$,
	there exist $\xi \in T_z\Omega_{p,q}^{1,2}$ and $c \in \mathbb{R}$
	such that \eqref{eq:Npqmanifold-proof1} holds,
	hence $\mathcal{N}_{p,q}$ is a $C^1$ submanifold of $\Omega^{1,2}_{p,q}$.

	By the previous part of the proof,
	for all $z\in \mathcal N_{p,q}$,
	$T_z\mathcal{N}_{p,q}$ is identified with the set 
	of all $\zeta$ such that $\de F(z)[\zeta] \in T_{F(z)}\mathcal{C}$.
	Then, \eqref{eq:TzNpq} follows from \eqref{eq:Npqmanifold-diffF}
	and the fact that $T_{F(z)}\mathcal{C}$
	is identified with the set of constant functions on $[0,1]$.

	It remains to show that $\mathcal{N}_{p,q}$ is closed.
	Let $(z_n)_n \subset \mathcal{N}_{p,q} \subset \Omega^{1,2}_{p,q}$ be a sequence converging to 
	$z\in \Omega^{1,2}_{p,q}$.
	Up to considering a subsequence, we have that
	$N(z_n,\dot{z}_n)$ converges pointwise to $N(z,\dot{z})$,
	so $N(z,\dot{z})$ is constant a.e. on $[0,1]$, i.e.  $z \in \mathcal{N}_{p,q}$.

\end{proof}

\begin{lemma}\label{lem:directsum}
	For each $z\in \mathcal N_{p,q}$,
	$T_z\Omega_{p,q}^{1,2} = \mathcal{W}_z \oplus T_z\mathcal{N}_{p,q}$.
\end{lemma}
\begin{proof}
	It is enough to show that
	for each 
	$\zeta \in T_z\Omega^{1,2}_{p,q}$
	there exists
	$\mu \in W^{1,2}_0([0,1],\mathbb{R})$ such that
	\[
		\xi:=\zeta - \mu K(z)\in T_z\mathcal{N}_{p,q}.
	\]
	By \eqref{eq:TzNpq}, this amounts to prove that there exist
	$\mu \in W^{1,2}_0([0,1],\mathbb{R})$ and a constant $c\in\R$ such that 
	\[
		\partial_{x}N(z,\dot{z})[\xi]+Q(\dot\xi) = c,
		\quad \text{a.e. on }[0,1],
	\] 
	which is equivalent to 
	\begin{equation}
		\label{eq:directsum-proof}
		\partial_{x}N(z,\dot{z})[\zeta]+Q(\dot\zeta) 
		-
		\mu\big(\partial_{x}Q(\dot{z}, K)+Q(\dot K)\big)-\mu'Q(K(z)) = c,
	\end{equation}
	a.e. on $[0,1]$. Arguing as in the proof of Proposition~\ref{prop:Npqmanifold},
	we see  that \eqref{eq:directsum-proof} admits a solution
	$\mu\in W^{1,2}_0([0,1],\mathbb{R})$ for a certain constant $c$,
	and we are done.
\end{proof}
\begin{definition}
	The {\em reduced action functional} $\mathcal J$ is the restriction of the functional $\mathcal A$ to the manifold $\mathcal{N}_{p,q}$, i.e. $\mathcal J\colon \mathcal{N}_{p,q} \to \mathbb{R}$, $\mathcal J = \mathcal A \big|_{\mathcal{N}_{p,q}}$.
\end{definition}
\begin{remark}
	Being $\A\in C^1(\Omega^{1,2}_{p,q})$, we get that $\J$ is $C^1$ on $\mathcal N_{p,q}$ as well.
\end{remark}

\begin{teo}
	\label{teo:punticriticiLiberi}
	A curve $z \in \Omega_{p,q}^{1,2}$ is a critical point for $\A$
	if and only if $z \in \mathcal{N}_{p,q}$ and 
	$z$ is a critical point for $\J$.
\end{teo}
\begin{proof}
	Let us assume that $z$ is a critical point for $\A$.
	Then 
	$\diff\A(z)[\xi] = 0$ for all $\xi \in \mathcal{W}_z \subset T_z\Omega^{1,2}_{p,q}$
	and by Proposition \ref{prop:NpqCost}
	we have $z \in \mathcal{N}_{p,q}$.
	Since $T_z\mathcal{N}_{p,q}\subset T_z\Omega^{1,2}_{p,q}$,
	\[
		\diff\J(z)[\xi] = \diff\A(z)[\xi] = 0,
		\quad\forall \xi \in T_z\mathcal{N}_{p,q},
	\]
	so $z$ is a critical point for $\J$.

	Now, let us assume that $z\in \mathcal{N}_{p,q}$ and
	$z$ is a critical point for $\J$.
	By Lemma~\ref{lem:directsum},
	for every $\zeta \in T_z\Omega_{p,q}^{1,2}$ there
	exist $\xi \in T_z\mathcal{N}_{p,q}$
	and $\psi \in \mathcal{W}_z$ such that
	$\zeta = \psi + \xi$.
	By Proposition \ref{prop:NpqCost}, we have
	$\diff\A(z)[\psi] = 0$, 
	while
	$\diff\A(z)[\xi] = \diff\J(z)[\xi] = 0$
	because $z$ is a critical point for $\J$.
	Therefore, $\diff\A(z)[\zeta] = \diff\A(z)[\psi] + \diff\A(z)[\xi] = 0$,
	namely $z$ is a critical point for $\A$.
\end{proof}

\section{Lower boundedness and Palais-Smale condition for the reduced action}
Let us give a  condition on the manifold $\mathcal N_{p,q}$
implying  that $\J$ is bounded from below and satisfies the Palais-Smale condition.
For every $c\in \mathbb{R}$, we denote by $\J^c$ the sublevel of $\J$, namely
\[
	\J^c := \left\{z \in \mathcal{N}_{p,q}: \J(z)\le c\right\}.
\]
\begin{definition}\label{c-boundedness}
	We say that $\mathcal N_{p,q}$ is {\em $c$-bounded} if $\J^c\neq\emptyset$ and 
	\[
		N_c := 	\sup_{z\in\J^c}|N(z,\dot z)|<+\infty.
	\]
\end{definition}

\begin{proposition}
	\label{prop:Jbounded}
	Under Assumptions~\ref{ass:L}---\ref{ass:bounds}, let $c\in\R$ such that $\mathcal N_{p,q}$ is $c$-bounded.
	Then, $\J$ is bounded from below.
\end{proposition}
\begin{proof}
	By \eqref{L} and \eqref{quadgrowth}, we obtain
	\begin{multline}
		\label{Jbelow}
		\J(z) = \int_{0}^{1} L(z,\dot{z})\de s\\
		\geq \inte
		\bigg(\frac{\lambda(z)}{4}\norm{\dot{z}}^2
			+L(z,0)-\frac{1}{\lambda(z)}\|\partial_vL(x,0)\|^2
		\bigg)\de s
		+ \inte\frac{Q^2(\dot{z})}{Q(K(z))}\de s
	\end{multline}
	
		Since $\mathcal{N}_{p,q}$ is $c$-bounded
		and using \eqref{dbounded},	
	for every $z \in \J^c$ we have
	\begin{equation}
		\label{Qbounded}
		Q^2(\dot z)=\big(N(z,\dot z)-d(x)\big)^2\leq 2(N_c^2+k_2^2)
	\end{equation}
	thus, using \eqref{eq:boundC2} and  \eqref{eq:boundNoether}
	we have
	\[
		\J(z) \ge c_2-\frac{c_3^2}{c_1}- \frac{2(N_c^2+k_2^2)}{k_1,}
	\]
	and the thesis follows.
\end{proof}

We show now that $c$-boundedness and Assumptions~\ref{ass:L}---\ref{ass:bounds}
imply a compactness condition for the sublevels of $\J$.  
\begin{lemma}
	\label{lem:zH1bounded-Jc}
	Let $c \in \mathbb{R}$ be such that
	$\mathcal{N}_{p,q}$ is $c$-bounded.
	If Assumptions~\ref{ass:L}---\ref{ass:bounds} hold,
	then every sequence $(z_n)_n \subset \J^c$ admits a 
	uniformly convergent subsequence.
\end{lemma}
\begin{proof}
	From \eqref{Jbelow} and  Assumption \ref{ass:bounds},
	if 
	$\mathcal{N}_{p,q}$ is $c$-bounded
	we have
	\[
		c \ge \J(z_n)\ge
		\frac{c_1}{4} \int_{0}^{1}\norm{\dot{z}_n}^2 \de s
		+ c_2  -\frac{c_3^2}{c_1}-\frac{2(N_c^2+k_2^2)}{k_1},
	\]
	hence the sequence $\norm{\dot{z}_n}$ is bounded in $L^2([0,1])$.
	Then,
	denoting by $d_g$ the distance induced by the metric $g$,
	by the Cauchy-Schwarz inequality
	we have
	\[
		d_g (z_n(s_2), z_n(s_1)) \leq
		\int_{s_1}^{s_2} \|\dot z_n\| \de s\leq
		|s_2-s_1|^{1/2} \left( 	\int_0^1 \|\dot z_n\|^2\de s\right)^{1/2},
	\]
	for all $0\leq s_1 \leq s_2\leq 1$.
	Thus, $(z_n)$ is uniformly bounded and uniformly equicontinuous and,
	being $(M,g)$ complete,
	by the Ascoli-Arzel\`a theorem there exists a uniformly convergent subsequence.
\end{proof}

\begin{definition}
	A sequence $(z_n)_n\subset \J^c$ is said
	a Palais-Smale sequence for $\J$ 
	if $\diff\J(z_n)\to 0$ strongly.
	We say that $\J$ satisfies the Palais-Smale condition on $\J^c$
	if every Palais-Smale sequence $(z_n)_n\subset \J^c$ 
	admits a strongly converging subsequence. 	
\end{definition}

\begin{remark}\label{psbounded}
	We point out that
	if Assumptions~\ref{ass:L}---\ref{ass:bounds} hold and
	$\mathcal{N}_{p,q}$ is $c$-bounded,
	then $\J$ is bounded on any sequence
	$(z_n)\subset \J^c$ by Proposition~\ref{prop:Jbounded},
	as it is required in the usual definition
	of the Palais-Smale condition. 
\end{remark}

\begin{teo}
	\label{PS}
	Under Assumptions~\ref{ass:L}--\ref{ass:bounds},
	assume also that $\mathcal{N}_{p,q}$ is $c$-bounded.
	Then $\J$ satisfies the Palais-Smale condition on $\J^c$.	
\end{teo}
\begin{proof}
	Let $(z_n)_{n}\subset \J^c$ be a Palais-Smale sequence for $\J$.
	By Lemma \ref{lem:zH1bounded-Jc},
	there exists a subsequence,
	still denoted by $(z_n)_n$, which uniformly converges
	to a continuous curve $z\colon[0,1]\to M$ such that $z(0)=p$ and $z(1)=q$.

	Let us now notice that by Lemma~\ref{lem:directsum},
	and taking into account that the supports of the curves $z_n$
	are in a compact subset of $M$, 
	if $\zeta_n\in T_{z_n}\Omega^{1,2}_{p,q}$ is bounded in $H^1$ norm then
	there exist two  bounded sequences $\xi_n \in T_{z_n}\mathcal{N}_{p,q}$ 
	and $\mu_n \in H^1_0([0,1],\mathbb{R})$
	such that $\zeta_n = \xi_n + \mu_n K_{z_n}$.
	By Proposition~\ref{prop:NpqCost} and since $z_n$ is a Palais-Smale sequence,
	we obtain
	\[
		\diff\A(z_n)[\zeta_n]=\diff\A(z_n)[\xi_n]+\diff\A(z_n)[\mu_nK_{z_n}]=
		\diff\J(z_n)[\xi_n]\to 0.
	\] 	
	We  now apply a localization argument as in \cite[Appendix A]{Abbondandolo2007};
	thus, we can assume that
	the Lagrangian $L$ is defined on $[0,1]\times U \times \R^{m+1}$,
	with $U$ an open neighborhood of $0$ in $\R^{m+1}$.
	Moreover, we can
	identify $(z_n)_n$ with a sequence in the Sobolev space
	$H^1([0,1],U)$.
	By Lemma \ref{lem:zH1bounded-Jc}, taking into account that the
	curves $z_n$ have fixed end-points, we get that $(z_n)_n$ is bounded in
	$H^1([0,1],U)$ and so it admits a subsequence, still denoted by
	$(z_n)$, which weakly and uniformly converges to a curve $z\in
	H^1([0,1],\R^{m+1})$ which also satisfies the same fixed end-points boundary
	conditions. 
	Thus, being $z_n-z$ bounded in $H^1$, we have $\diff\A (z_n)[z_n-z]\to 0$, i.e.
	\begin{multline*}
		\int_0^1 \partial_x L_c(z_n,\dot z_n)[z_n-z]\de s 
		+ \int_0^1 \partial_v L_c(z_n,\dot z_n)[\dot z_n-\dot z]\de s\\
		-\inte\frac{2Q(\dot z_n)\partial_{x}Q(\dot z_n, z_n-z)}{\Lambda(z_n)}\de s
		-\inte \frac{2Q(\dot z_n)Q(\dot z_n-\dot z)}{\Lambda(z_n)}\de s\\
		+\inte
		\frac{Q^2(\dot z_n)\diff\Lambda(z_n)[\dot z_n-\dot z]}{\Lambda^2(z_n)}\de s
		\longrightarrow 0,
	\end{multline*}
	where $\Lambda(x):=-Q_x(K)$.
	From \eqref{partialxpinched},
	\[
		\big|\partial_x L_c(z_n,\dot z_n)[z_n- z]\big |\leq
		C(z_n)\big(\|\dot z_n\|^2+1\big)\|z_n-z\|,
	\]
	thus, recalling that $C$ is continuous and $z_n-z$ uniformly converges to
	$0$, the first integral in the above expression converges to $0$.
	Since the sequence  $Q(\dot z_n )$ is uniformly bounded on $[0,1]$
	(recall \eqref{Qbounded}) and,
	from \eqref{eq:boundNoether},
	$0<1/\Lambda (z_n)<1/k_1$,
	the third term above converges to $0$ because
	$\dot z_n$ is bounded in $L^1$ and $z_n-z\to 0$ uniformly.
	Analogously the fourth term goes to $0$
	since $z_n$ converges uniformly to $z$ and $\dot z_n-\dot z\to 0$ weakly in $H^1$.
	For estimating the fifth term, taking into account that $Q^2(\dot z_n)$ is uniformly bounded on $[0,1]$, 
	we observe that $\text{d}\Lambda (z_n)\to \text{d}\Lambda (z)$
	in operator norm and then 
	\begin{multline*}
		\inte\frac{\diff\Lambda(z_n)[\dot z_n-\dot z]}{\Lambda^2(z_n)}\de s\\
		=\inte\frac{\big(\diff\Lambda(z_n)-\diff\Lambda(z)\big)
		[\dot z_n-\dot z]}{\Lambda^2(z_n)}\de s
		+\inte\frac{\diff\Lambda(z)[\dot z_n-\dot z]}{\Lambda^2(z_n)}\de s,
	\end{multline*}
	and both the above integrals goes to $0$, because $\dot z_n-\dot z$,
	in the first one, is bounded in $L^1$ and, in the second one,
	weakly converges to $0$ in $H^1$.
	Thus, we have obtained that  
	\begin{equation}
		\label{almostdone}
		\int_0^1 \partial_v
		L_c(z_n,\dot z_n)[\dot z_n-\dot z]\de s\longrightarrow 0.
	\end{equation}
	Using that $z_n$ pointwise converges to $z$ and $\dot z_n$ is bounded in $L^1$, from \eqref{partialvpinched} and Lebesgue's dominated convergence theorem we get
	\[
		\int_0^1 \partial_vL_c(z_n,\dot z)[\dot z_n-\dot z]\de s-\int_0^1 \partial_v
		L_c(z,\dot z)[\dot z_n-\dot z]\de s\longrightarrow 0.
	\] 
	As $\dot z_n-\dot z\to 0$ weakly in $H^1$,
	also $\inte\partial_v L_c(z,\dot z)[\dot z_n-\dot z]\de s\rightarrow 0$,
	and then from the above limit 
	\begin{equation}
	\label{almostdone2}
	\int_0^1 \partial_vL_c(z_n,\dot z)[\dot z_n-\dot z]\de s\longrightarrow 0.
	\end{equation}
	From \eqref{eq:boundC1},  \eqref{almostdone} and \eqref{almostdone2} we then get 
	\[
		\frac{c_1}{4}\inte |\dot z_n-\dot z|^2\de s\leq \int_0^1 \big(\partial_v
			L_c(z_n,\dot z_n)-\partial_v
		L_c(z_n,\dot z)\big)[\dot z_n-\dot z]\de s\longrightarrow 0,
	\]
	which implies that $z_n\to z$ strongly in $H^1$.
	Moreover, there exists a subsequence such that $\dot z_n(s)\to \dot z(s)$
	a.e. on $[0,1]$ and then 
	\[
		N\big(z_n(s),\dot z_n(s)\big)\to N\big(z(s), \dot z(s)\big),
		\quad\text{a.e. on $[0,1]$},
	\]
	so that also $N(z,\dot z)$ is constant a.e. on $[0,1]$,
	i.e. $z\in \mathcal N_{p,q}$ as required.
\end{proof}
From Propositions~\ref{prop:Jbounded} and Theorem~\ref{PS}, $\J$ is bounded from below and satisfies the Palais-Smale condition on $\J^c$.
Since $\mathcal N_{p,q}$ is only a $C^1$ submanifold of $\Omega_{p,q}^{1,2}$
(recall Proposition~\ref{prop:Npqmanifold}) then the exponential  map of its
infinite dimensional Riemannian structure is not well-defined, and we cannot
invoke  Ekeland's variational principle to conclude that a minimizer of $\J$
exists (see \cite[Proposition 5.1]{Ekelan74}).
Anyway, from \cite[Theorem 3.1]{szulkin1988}
(which, nevertheless,  is based on Ekeland's variational principle)
or  as a straightforward  consequence of the
noncritical interval theorem  (see \cite[Theorem (2.15)]{corvellec1993}),
we actually get the existence of a  minimizer of $\J$. Summing up, we
have the following result:
\begin{teo}
	\label{teo:main}
	Let $L\colon TM \to \mathbb{R}$ be an indefinite Lagrangian satisfying
	Assumptions~ \ref{ass:L}--\ref{ass:bounds}.
	Assume also that $\mathcal{N}_{p,q}$ is $c$-bounded, for some $c\in \R$.
	Then there exists a curve $z \in \mathcal N_{p,q}$
	which minimizes $\J$ and it is then
	a critical point of $\A$ on $\Omega_{p,q}^{1,2}$.
\end{teo}

\begin{remark}
	The critical points of $\A$ on $\Omega_{p,q}^{1,2}$,
	whose existence is ensured by Theorem \ref{teo:main},
	satisfy the Euler-Lagrange equation \eqref{eq:Euler-Lagrange} in weak sense.
	We will show in Appendix~\ref{sec:regularity}
	that they also satisfy it in classical sense.
\end{remark}

\section{Multiplicity of critical points}
In this section we obtain a multiplicity result for critical points
of the functional $\A$ by using  Ljusternik-Schnirelmann theory,
provided that $M$ is a not contractible.
Let us recall the definition of Ljusternik-Schnirelmann category.
Let $A$ be a non-empty subset of a topological space $B$;
the {\em Lusternik-Schnirelman category} of a $A$,
denoted by $\mathrm{cat}_B(A)$,
is the least integer $n$ such
that $A$ can be covered $n$ closed contractible (in $B$) subsets of $B$.
If no such a number exists then $\cat_B(A)=+\infty$.
If $A = \emptyset$, we set $\cat_B{A} = 0$.
We denote $\cat_B(B)$ with $\cat(B)$.

By \cite[Proposition 3.2]{fadell1991},
we know that if $M$ is a non-contractible manifold then
$\cat(\Omega_{p,q}^{1,2}) = + \infty$.
This  fact can be exploited together with the following proposition,
which is a straightforward corollary of 
\cite[Theorem (3.6)]{corvellec1993}
and allows to prove the multiplicity of critical points for 
a functional of class $C^1$ defined on a manifold with the same regularity,
as it is in our setting.

\begin{teo}[Corvellec-Degiovanni-Marzocchi]\label{CDM}
	Let $\mathcal M$ be a (possibly infinite dimensional)
	$C^1$  Riemannian manifold
	and $f:\mathcal M \to \R$ be a bounded from below $C^1$ functional
	satisfying the Palais-Smale condition.

	Then $f$ has at least $\cat(\mathcal{M})$ critical points.
	Moreover, if $\cat(\mathcal{M})=+\infty$ then $\sup f=+\infty$
	and there exists a sequence $(c_m)_m$ of critical values such that $c_m\to +\infty$.    
\end{teo}
\begin{remark}
	Actually \cite[Theorem (3.6)]{corvellec1993} is stated for a continuous
	functional $f$ on a complete metric space $X$ with a critical point defined by
	using the notion of {\em weak slope} introduced in \cite{Degiovanni1994}.
	Points with vanishing weak slope are standard critical  points if $f$ is  a
	$C^1$ functional on a Riemannian manifold.
	The metric space must also be  {\em weakly locally contractible},
	meaning that  each $x\in X$ admits a neighborhood
	contractible in $X$.
	Notice that if $X$ is weakly locally contractible then,
	for each $x\in X$, $\cat_X(\{x\})=1$.
	A $C^1$  Riemannian manifold  is clearly
	weakly locally contractible (it is enough to take a small neighborhood of $x$
	diffeomorphic to a ball  in the model Hilbert space).
	Thus, for example, both $\Omega_{p,q}^{1,2}$ and $\mathcal N_{p,q}$
	are  weakly locally contractible,
	the latter a fortiori being also a strong deformation retract of
	$\Omega_{p,q}^{1,2}$ if  $K$ is complete
	(see Proposition~\ref{prop:deformation-retract}).
	Finally, we notice that in \cite{corvellec1993} the definition of 
	Ljusternik-Schnirelman category is
	given with open coverings instead of closed one.
	This is equivalent to the definition with closed coverings in every ANR space;
	since metrizable manifolds are ANR (see \cite[Theorem 5]{Palais66a}),
	the two definitions are then equivalent for $\mathcal N_{p,q}$. 
\end{remark}

Let us now state the main result of this section.
\begin{teo}
	\label{teo:multiplicity}
	Let $M$ be a non-contractible manifold and $L\colon TM \to \mathbb{R}$
	a Lagrangian that satisfies Assumptions \ref{ass:L}--\ref{ass:bounds}.
	If $K$ is a complete vector field and
	$\mathcal{N}_{p,q}$ is $c$-bounded for all $c \in \mathbb{R}$,
	then there exists a sequence $(z_n)_{n\in\mathbb{N}}\subset \Omega_{p,q}^{1,2}$
	of critical points of $\A$ such that
	\[
		\lim_{n\to\infty}\A(z_n) = + \infty.
	\]
\end{teo}

Like  in the existence result given in Theorem \ref{teo:main},
we cannot work directly on $\Omega_{p,q}^{1,2}$ to prove Theorem~\ref{teo:multiplicity},
where $\A$ is not bounded from below and does not satisfy the Palais-Smale condition,
but we have to restrict our analysis on $\mathcal{N}_{p,q}$.

Let us first show that when $K$ is complete then $\mathcal N_{p,q}$
is a strong deformation retract of $\Omega^{1,2}_{p,q}$
(so that the Ljusternik-Schnirelmann category is preserved), namely there exists a homotopy
$H\colon \Omega_{p,q}^{1,2}\times [0,1] \to \Omega_{p,q}^{1,2}$
	such that, for all
	$z \in \Omega_{p,q}^{1,2}$,
	$w \in \mathcal{N}_{p,q}$
	and $t \in [0,1]$,
we have $H(z,0) = z$, $H(z,1)\in \mathcal{N}_{p,q}$ and $H(w,t) = w$.
Next proposition extends \cite[Proposition 5.9]{giannoni1999}
from stationary Lorentzian manifold to our setting.
\begin{proposition}
	\label{prop:deformation-retract}
	Assume that $K$ is a complete vector field, then  $\mathcal{N}_{p,q}$ is a
	strong deformation retract of $\Omega_{p,q}^{1,2}$.
\end{proposition}

In the proof of  Proposition \ref{prop:deformation-retract},
it will be useful  the following preliminary result.
\begin{lemma}
	\label{lem:F-map}
	Let the vector field $K$ be complete 
	and let $\psi\colon \mathbb{R}\times M \to M $
	be its flow.
	Then, for every $z \in \Omega^{1,2}_{p,q}$ there exists a uniquely defined
	function $\phi\in H^{1}_0([0,1], \R)$
	such that 
	\begin{equation}
		\label{eq:def-w-flow}
		\psi\big(\phi(\cdot),z(\cdot)\big) \in \mathcal{N}_{p,q}.
	\end{equation}
	Moreover, defining $\Psi\colon \Omega_{p,q}^{1,2} \to \mathcal{N}_{p,q}$
	as 
	\[		\big(\Psi(z)\big)(s):= \psi\big(\phi(s),z(s)\big),
	\]
	the function $\Psi$ is $C^1$.
\end{lemma}
\begin{proof}
	Let $z\in \Omega_{p,q}^{1,2}$
	and, for each $\phi\in H^{1}_0([0,1], \R)$,
	let us denote by $w\colon [0,1] \to M$ the curve
	\begin{equation}
		\label{eq:w}
		w(s) = \psi(\phi(s),z(s))
	\end{equation}

	We want to find $\phi\in H^{1}_0([0,1], \R)$ such that  $w \in \mathcal{N}_{p,q}$, hence
	$w(0) = p$, $w(1) = q$ and 
	\begin{equation}
		\label{NC}
		N(w,\dot{w}) = C,\quad\text{a.e. on $[0,1]$,}
	\end{equation}
	for some constant $C \in \mathbb{R}$.
	By differentiating \eqref{eq:w}, we get 
	\[
		\dot{w}(s) = 
		\partial_t\psi(\phi(s),z(s))\phi'(s)
		+ \partial_x\psi(\phi(s),z(s))[\dot{z}(s)].
	\]
	Substituting this expression in \eqref{NC} and  recalling that $N=Q+d$, we get
	\begin{multline}
		\label{eq:boh1}
		\phi'(s)Q_{w(s)}\big(\partial_t\psi\big(\phi(s),z(s)\big)\big)
		+ Q_{w(s)}\big(\partial_x\psi(\phi(s),z(s))[\dot z(s)]\big)+d\big(w(s)\big)\\
		=\phi'(s)Q_{w(s)}\big(\partial_t\psi\big(\phi(s),z(s)\big)\big)
		+N\big(w(s), \partial_x\psi(\phi(s),z(s))[\dot{z}]\big)=C
	\end{multline}
	which, for each $z\in \Omega^{1,2}_{p,q}$,
	can be seen as a differential equation for $\phi$.

	Let us  rewrite \eqref{eq:boh1} in order to get a simpler equation.
	Since $\psi=\psi(t,x)$ is the flow generated by $K$, we have
	\begin{equation}
		\label{eq:dtpsi-K}
		\partial_t\psi(\phi,z) = K(\psi(\phi,z)) = K(w).
	\end{equation}	
	Using the group property
	$\psi(t_1,\psi(t_2,x)) = \psi(t_1 + t_2,x)$,
	and \eqref{eq:dtpsi-K} we also obtain
	\begin{equation}
		\label{eq:dxpsi-K}
		\partial_x\psi(\phi,z)[K(z)] = K(w).
	\end{equation}
	Moreover, recalling \eqref{eq:L-Kinvariant},
	for every $v \in T_{z(s)}M$ we have
		\[
			L(z(s),v) = L(w(s),\partial_x\psi(\phi(s),z(s))[v]),
		\]
		thus
		\[
			\partial_vL(z(s),v)[K] = 
			\partial_vL(w(s),\partial_x\psi\big(\phi(s),z(s))[v]\big)\big[\partial_x\psi(\phi(s),z(s))[K]\big].
	\]
	By \eqref{noether} and  \eqref{eq:dxpsi-K}, the last equality becomes
	\begin{equation}
		\label{eq:Q_w-eq-Q_x}
		N(z(s),v) = N\big(w(s),\partial_x\psi(\phi(s),z(s))[v]\big),
		\quad \forall v \in T_{z(s)}M.
	\end{equation}
	Substituting $v$ with $\dot z$ in \eqref{eq:Q_w-eq-Q_x},
	we get
	\begin{equation}
		\label{pezzo2}
		N(z(s),\dot z(s)) = N\big(w(s),\partial_x\psi(\phi(s),z(s))[\dot z]\big).
	\end{equation}
	Recalling that $Q(K)$ is invariant by the flow
	(see the proof of Proposition~\ref{prop:Lproperties}-\eqref{LflowK}),
	by \eqref{eq:dtpsi-K} we have
	\begin{equation}
		\label{pezzo1}
		Q_{w(s)}\big(\partial_t\psi\big(\phi(s),z(s)\big)\big)=Q_{w(s)}(K)
		=Q_{z(s)}(K).
	\end{equation}
	Thus, from \eqref{pezzo2} and \eqref{pezzo1}, \eqref{eq:boh1} becomes
	\[
		\phi'Q(K(z))+N(z,\dot z)=C
	\] 
	and since by
	Assumption~\ref{ass:L}, $Q(K(z))$ is   different from $0$,
	we get
	\begin{equation}
		\label{eq:ODE-phi}
		\phi' = \frac{C - N(z,\dot{z})}{Q(K(z))}.
	\end{equation}
	Hence, $\phi$ can be obtained as the solution of \eqref{eq:def-C-ODE-phi}  with
	initial condition $\phi(0) = 0$ and,
	by setting $\int_{0}^{1}\phi'(s)\diff s = 0$,
	we can ensure that $\phi(1) = 0$ by taking
	\begin{equation}
		\label{eq:def-C-ODE-phi}
		C = \left( \int_{0}^{1}\frac{N(z,\dot{z})}{Q(K(z))}\diff s \right)
		\left( \int_{0}^{1}\frac{\de s}{Q(K(z))} \right)^{-1}.
	\end{equation}
	The fact that $\Psi$ is $C^1$ is a simple consequence of the $C^1$-regularity of $N$ and \eqref{eq:ODE-phi}.
\end{proof}

\begin{proof}[Proof of Proposition \ref{prop:deformation-retract}]
	By Lemma \ref{lem:F-map},
	for every $z \in \Omega_{p,q}^{1,2}$, we consider $\phi \in H^{1}_0([0,1],\R)$
	(depending on $z$ and univocally defined as shown in Lemma~\ref{lem:F-map}),  
	such that \eqref{eq:def-w-flow} holds.
	Then, let us define $H\colon \Omega_{p,q}^{1,2}\times[0,1] \to \Omega_{p,q}^{1,2}$
	as
	\[
		H(z,t) := \psi(t\phi,z).
	\]
	Notice that $H(\cdot,0)$ is the identity map on $\Omega_{p,q}^{1,2}$,
	and 
	$H(\Omega_{p,q}^{1,2},1)\subset \mathcal{N}_{p,q}$.
	If $w \in \mathcal{N}_{p,q}$, then $N(w,\dot{w})$ is constant
	and recalling that $\phi$ satisfies \eqref{eq:ODE-phi}
	with $C$ given by \eqref{eq:def-C-ODE-phi},
	we get that
	the corresponding $\phi$ is the zero  function,
	hence $ H(w,t) = w$ for all $t \in [0,1]$.
\end{proof}

We can  now  prove Theorem~\ref{teo:multiplicity}.

\begin{proof}[Proof of Theorem \ref{teo:multiplicity}]
	Since the Ljusternik-Schnirelmann category is a homotopy invariant,
	by Proposition \ref{prop:deformation-retract} and \cite[Proposition 3.2]{fadell1991},
	we have
	$\cat (\mathcal{N}_{p,q}) = \cat (\Omega_{p,q}^{1,2}) = + \infty$.
	From Theorem~\ref{PS}, $\J$ 
	satisfies the Palais-Smale condition on $\mathcal J^c$ for every $c\in \R$ and,
	then, on $\mathcal N_{p,q}$ (recall Remark~\ref{psbounded}).
	Hence, by Theorem~\ref{CDM}, there exists a sequence $(z_n)_n \subset  \mathcal{N}_{p,q}$
	of critical points of $\J$ such that $\J(z_n)\to +\infty$.
	By Theorem \ref{teo:punticriticiLiberi},
	every critical point of $\J$ is a critical point of $\A$,
	and $\mathcal A(z_n)=\J(z_n)$.
\end{proof}

\section{$c$-precompactness and $c$-boundedness}\label{pseudocoercivity} 
In light of  Theorems~\ref{teo:main} and \ref{teo:multiplicity},
it becomes important to give conditions ensuring the $c$-boundedness of $\mathcal N_{p,q}$. 
We firstly need the following definition, introduced in \cite{giannoni1999}.
\begin{definition}
	\label{cprecompact}
	Let $c$ be a real number.
	The set $\mathcal{N}_{p,q}$ is said to be 
	{\em $c$-precompact} if 
	every sequence $(z_n)_n \subset \J^c$
	has a uniformly convergent subsequence.
	We say that $\J$ is {\em pseudocoercive} if $\mathcal N_{p,q}$ is $c$-precompact for all $c\in\R$.
\end{definition}
We are going to show that $c$-boundedness and $c$-precompactness
are essentially equivalent properties for Lagrangians admitting
a local expression of ``product'' type
(see \eqref{eq:splitting-Lchart} below).
As a first step, we notice that Lemma \ref{lem:zH1bounded-Jc} immediately gives one of the implications in the equivalence.
\begin{proposition}
	\label{prop:cboundTOcprecompact}
	Let Assumptions~\ref{ass:L}---\ref{ass:bounds} hold.
	If $\mathcal{N}_{p,q}$ is $c$-bounded,
	then it is $c$-precompact.
\end{proposition}
The converse implication holds if $L$
admits a local structure of the type in \eqref{eq:gen-L},
so we give the following definition.

\begin{definition}\label{def:loc}	
	We say that $L$ admits a {\em stationary product type local structure} if 
	for every point $p \in M$ there exist an open precompact neighborhood
	$U_p \subset M$ of $p$,
	a manifold with boundary $S_p$,
	an open interval $I_p = (-\epsilon_p,\epsilon_p)\subset \mathbb{R}$,
	and a diffeomorphism $\phi\colon S_p \times I_p \to U_p$
	such that, named $t$ the natural coordinate of $I_p$,
	\[
		\phi_*(\partial_t) = K\big|_{U_p},
	\]
	and for all $\big((x,t),(\nu,\tau)\big) \in T(S_p \times I_p)$ we have
	\begin{equation}
		\label{eq:splitting-Lchart}
		L \circ \phi_*\big((x,t),(\nu,\tau)\big)
		= L_0(x,\nu) + 2 \big(\omega(\nu)+d(x)/2\big)\tau - \beta(x) \tau^2,
	\end{equation}
	where 
	\begin{itemize}
		\item $L_0 \in C^1(TS_p)$ is a
			Lagrangian on $S_p$ which satisfies
			\eqref{eq:genL-pinched}---\eqref{eq:genL-partialvpinched}
			with respect to the norm $\|\cdot\|_{S_p}$ of the
			metric induced on $S_p$ by the auxiliary Riemannian metric on $M$,
			and it is pointwise strongly convex,
			i.e. it satisfies \eqref{monotone} on $TS_p$
			(with $L_c$ replaced by $L_0$ and $\|\cdot\|$ by $\|\cdot\|_{S_p}$),
			for a continuous function $\lambda\colon S_p\to (0,+\infty)$;
		\item $\omega$ is a $C^1$ one-form on $S_p$;
		\item $d\colon S_p \to \mathbb{R}$ is a $C^1$ function;
		\item $\beta\colon S_p \to (0,+\infty)$ is a positive $C^1$ function.
	\end{itemize}

\end{definition}

Notice that Definition~\ref{def:loc} is satisfied for  
$L(v)=g_L(v,v)$, where $g_L$ is a $C^1$ Lorentzian metric on $M$ having  a
timelike Killing vector field $K$;
in such a case $S_p$ is a spacelike hypersurface in $M$,
$L_0$ is the Riemannian metric induced on it by $g_L$,
$\beta(x)=-g_L(K_x,K_x)$ and $\omega$ is the one-form metrically equivalent to
the orthogonal projection of $K$ on $TS_p$
and $d \equiv 0$
(see, e.g., \cite[Appendix C]{giannoni1999}).
The next result shows that it is satisfied as well by a Lagrangian
fulfilling Assumptions~\ref{ass:L}-\ref{ass:L1}.
\begin{proposition}
	\label{productype}
	Let $L:TM\to \R$ satisfy Assumptions~\ref{ass:L} and \ref{ass:L1}.
	Then it admits a stationary product type local structure.
\end{proposition}
\begin{proof}
	Let us denote by $\D$ the distribution in $TM$
	generated by the kernel of $Q$,
	i.e. for all $z\in M$, $\D_z=\ker Q_z$. Notice that by \eqref{negativeQK}, $\D$ has constant rank equal to $m$ (recall that $\dim(M)=m+1$).
	Let $ \bar z\in M$ and $S_{\bar z}$ be a smooth hypersurface (with boundary) in $M$
	such that $ \bar z\in S_{\bar z}$ and
	$T_{\bar z}S_{\bar z}= \mathcal D_{\bar z}$. We  endow $S_{\bar z}$ with the Riemannian metric induced by the auxiliary Riemannian metric $g$ on $M$ and let us denote its norm with $\|\cdot\|_{S_{\bar z}}$.
	From \eqref{negativeQK}, up to shrink $S_{\bar z}$, we can assume that
	for all $x\in S_{\bar z}$, 	$K_x$ is transversal to $S_{\bar z}$,
	i.e. $T_x M=T_xS_{\bar z}\oplus [K_x]$.
	Using \eqref{L}, we get 
	\[
		\partial_{v}L(x,\nu)
		=\partial_{v}L_c(x,\nu)+\frac{2}{Q(K)}Q(\nu)Q_x,
	\]
	for all $(x,\nu)\in TS_{\bar z}$.
	In particular, $\partial_{v}L(\bar z,\nu)=\partial_{v}L_c(\bar z,\nu)$
	for all $\nu\in T_{\bar z}S_{\bar z}$.
	Considering a smaller hypersurface $S_{\bar z}$ such that 
	\begin{equation}
		\label{strong}
		\lambda_0:=\min_{x\in S_{\bar z}}\left(\lambda(x)+
		\frac{2}{Q(K)}\max_{\|\nu\|_{S_{\bar z}}=1} Q^2_x(\nu)\right)>0,
	\end{equation}
	for all $(x,\nu_1),(x,\nu_2)\in TS_{\bar z}$
	we have
	\begin{multline*}
	\big(\partial_vL(x,\nu_2)-\partial_vL(x,\nu_1)\big)[\nu_2-\nu_1]\\
	=\big(\partial_vL_c(x,\nu_1)-\partial_vL_c(x,\nu_2)\big)[\nu_2-\nu_1]\\
	+\frac{2}{Q(K)}Q^2(\nu_2-\nu_1)\geq \lambda_0\|\nu_2-\nu_1\|^2_{S_{\bar z}}.
	\end{multline*}
	Let $L_0=L|_{TS_{\bar z}}$;
	the above inequality gives then \eqref{monotone} for $L_0$ on $TS_{\bar z}$.
	Since $L_0(x, \nu)=L_c(x,\nu)+\frac{Q_x^2(\nu)}{Q_x(K_x)}$
	and $L_c$ satisfies \eqref{pinched}---\eqref{partialvpinched},
	we also have that $L_0$ satisfies
	\eqref{eq:genL-pinched}--\eqref{eq:genL-partialvpinched}.
	Let us now evaluate $\frac{\de}{\de s} L(x,y+s\tau K)$,
	for any $y\in T_x M$, $x\in M$, and $\tau \in\R$:
	\begin{multline*}
		\frac{\diff}{\diff s} L(x,y+s\tau K)
		=\partial_v L(x,y+s\tau K)[\tau K]
		=\tau\partial_v L(x,y+s\tau K)[K]\\
		=\tau \big(Q(y+s\tau K)+d(x)\big)
		=\tau \big(Q(y)+s\tau Q(K)+d(x)\big).
	\end{multline*}
	Hence, integrating w.r.t. $s$ between $0$ and $1$ we get
	\begin{equation}
		\label{Lyplus}
		L(x, y+\tau K)-L(x,y)=\tau \big(Q(y)+d(x)\big)+\frac{1}{2} \tau^2 Q(K).
	\end{equation}
	Let now $w\in T_x M$, $x\in S_{\bar z}$,
	and $w_S\in T_{x}S_{\bar z}$, $\tau_w\in\R$
	such that $w=w_S+\tau_wK_x$. 
	From \eqref{Lyplus} we get
	\begin{align}
		L(x, w)&=L(x, w_S+\tau_w K)\nonumber\\
				 &=L(x,w_S)+\tau_w \big(Q(w_S)+d(x)\big)+\frac 12 \tau_w^2 Q(K)\nonumber \\
				 &=L_0(x,w_S)+\tau_w \big(Q(w_S)+d(x)\big)+\frac 12 \tau_w^2 Q(K)
				 \label{asplitting}.
	\end{align}
	Thus, we get an expression of the type at the right-hand side of
	\eqref{eq:splitting-Lchart} on $S_{\bar z}$ by defining $\omega$
	as the one-form induced by $Q/ 2$ on $S_{\bar z}$ and
	$\beta(x):=-Q(K)/2$. 
	Since $L$ is invariant by the flow of $K^c$
	we then obtain \eqref{eq:splitting-Lchart} on $S_{\bar z}\times I_{\bar z}$,
	for some open interval $I_{\bar z}$ containing $0$,
	by taking $\phi$ as the restriction to $S_{\bar z}\times I_{\bar z}$ of the flow $\psi$ of $K$ adapted to $S_{\bar z}$,
	i.e. such that $S_{\bar z}=\psi(S_{\bar z}\times\{0\})$.
\end{proof}			
\begin{remark}
	Notice that if the distribution $\D$ generated by the kernel of $Q$ is
	integrable then we can take in the above proof $S_{\bar z}$ equal to an
	integral manifold of $\D$.
	In this case the local expression of $L$ simplifies to
	\[
		L \circ \phi_*\big((x,t),(\nu,\tau)\big)
		= L_0(x,\nu)+d(x)\tau - \beta(x) \tau^2.
	\]
	This can be considered as a generalization of
	the notion of a static Lorentzian metric to an indefinite
	Lagrangian admitting an infinitesimal symmetry satisfying
	Assumptions~\ref{ass:L}-\ref{ass:L1}
	(compare also with \cite{CapSta16, caponio2018}).
\end{remark}
By Proposition~\ref{productype} we obtain
the following generalization of \cite[Lemma 4.1]{giannoni1999}.
\begin{teo}
	\label{teo:cprecbound}
	Let Assumptions \ref{ass:L} and \ref{ass:L1} hold. 
	If $\mathcal N_{p,q}$ is $c$-precompact then it is $c$-bounded.
\end{teo}

\begin{proof}
	Let $(z_n)_n \subset \J^c$ be a sequence such that
	\[
		\lim_{n \to \infty} \abs{N(z_n, \dot z_n)}
		= \sup_{z \in \J^c}\abs{N(z,\dot z)}.
	\]
	Moreover, 
	let $(C_{z_n})_n \subset \mathbb{R}$ be the sequence of real numbers such that
	for all $n$
	\[
		C_{z_n}= \frac{1}{2}N\big(z_n(s), \dot{z}_n(s)\big), \quad \text{a.e. in }[0,1].
	\]
	To obtain the thesis, it suffices to prove that $C_{z_n}$ is bounded.
	Since $\mathcal N_{p,q}$ is $c$-precompact we can assume,
	up to pass to a subsequence, that 
	$z_n$ converges uniformly to a curve $z\in \J^c$.
	We can then assume that there exists a finite number of neighborhoods 
	$U_k$, with $k = 1,\dots,N$,
	that cover $z([0,1])$ such that, 
	for some finite sequence 
	$0 = a_0 < a_1 < \dots < a_N = 1$, 
	$z_n([a_{k-1},a_k])\subset U_k$, for all $n$ sufficiently large
	and for all $k = 1,\dots,N$.
	Moreover, by Proposition~\ref{productype}, 
	in each domain $U_k$ we can identify $L$ with $L\circ(\phi_k)_*$ so that $L$, evaluated along a curve $z(s)=\big(x(s),t(s)\big)$ contained in $U_k$, is given by
	\[
		L(z,\dot z)=L_{0,k}(x,\dot{x})+2\big(\omega_k(\dot{x})+d_k(x)/2\big)\dot{t}-\beta_k(x)\dot{t}^{2},
	\]
	(here we are not  writing the point where the one-forms $\omega_k$ are applied).
	Up to replace each $U_k$ by a precompact open subset, we can assume that
	\begin{align}
		\label{eq:def-D0}
				&\max_k(
				\norm{\omega_k})
				=\max_k
				\Big(
					\sup_{\substack{\norm{y} = 1\\y\in TU_k}}\abs{\omega_k(y)} 
				\Big)
				=D_0 < +\infty,
				\intertext{and}
				&\max_k\big(\sup_{x\in U_k}|d_k(x)|\big)=D_1<+\infty\label{D1},
	\end{align}
	reminding that $\norm{y} = \sqrt{g(y,y)}$,
	where $g$ is the auxiliary Riemannian metric.
	Analogously, we have 
	\[
		\Delta =\max_k\Big(\sup_{m_1,m_2 \in U_k}
		\abs{t^k(m_1) - t^k(m_2)}\Big)<+\infty
	\]
	and there also exist  two constants 	
	$\nu,\ \mu$ such that
	\[	0 < \nu \le \beta_k
		\le \mu,\quad\text{for all $k\in\{1,\ldots,N\}$.}
	\]
	In the following, we write $L_{0,k}$, $\omega_k$, $d_k$ and $\beta_k$ without the index $k$.
	In this local charts, let $z_n(s) = \big(x_n(s),t_n(s)\big)$.
	As for \eqref{eq:gen-L} and \eqref{noetherxample},
	we have $N(z_n, \dot z_n) = 2\omega(\dot{x}_n)-2\beta\dot{t}_n+d(x_n)$. 
	Hence,
	\begin{equation}
		\label{eq:tn_Cn}
		\dot{t}_n=\frac{\omega(\dot{x}_n)+d(x_n)/2- C_{z_n}}{\beta(x_n)}.
	\end{equation}
	Defining $T_n^k = t_n(a_k)- t_n(a_{k-1})$, we have
	\begin{equation}
		\label{eq:dn-def}
		T_n^k = \int_{a_{k - 1}}^{a_k} \dot{t}_n \diff s
		=
		\int_{a_{k-1}}^{a_k}\frac{\omega(\dot{x}_n)+d(x_n)/2- C_{z_n}}{\beta(x_n)}
		\diff s.
	\end{equation}
	Therefore, the quantity
	\[
		b_n^k :=
		\int_{a_{k-1}}^{a_{k}} \frac{\diff s}{\beta(x_n(s))}.
	\]
	is well-defined and finite.
	Moreover,
	\begin{equation}
		\label{eq:bound_bnk}
		\frac{a_k - a_{k-1}}{\mu}\le b_{n}^k \le \frac{a_k - a_{k-1}}{\nu}.
	\end{equation}
	From \eqref{eq:dn-def} we obtain 
	\begin{equation}
		\label{eq:C_zn1}
		C_{z_n}=
		\frac{1}{b_{n}^k}
		\left(
			\int_{a_{k-1}}^{a_{k}}\frac{\omega(\dot{x}_n)+d(x_n)/2}{\beta(x_n)}\diff s - T_n^k
		\right).
	\end{equation}	
	By \eqref{eq:def-D0},
	we have $\abs{w(\dot{x}_n)}\le D_0 \norm{\dot{x}_n}$.	
	As a consequence,
	using also \eqref{D1}, 
	$\abs{T_n^k}\le \Delta$
	and \eqref{eq:bound_bnk},
	from \eqref{eq:C_zn1} we have
	\begin{equation}
		\label{eq:C_zn2}
		\abs{C_{z_n}} <
		\frac{(D_0+D_1)\mu}{\nu(a_{k-1} - a_{k})}
		\int_{a_{k-1}}^{a_{k}}\norm{\dot{x}_n}\diff s
		+\frac{\mu\Delta}{a_{k-1}-a_{k}}.
	\end{equation}
	By \eqref{eq:C_zn2},
	to prove that $C_{z_n}$ is bounded, and thus to prove the theorem,
	it suffices to show that 
	\begin{equation}
		\label{eq:Fxn-bounded}
		\sup_{n}\int_{0}^{1}\norm{\dot{x}_n}\diff s < +\infty.
	\end{equation}
	To this end, recall that by \eqref{eq:tn_Cn} we have
	\begin{multline*}
	L\big((x_n,t_n), (\dot{x}_n,\dot{t}_n)\big)
	= L_0(x_n, \dot{x}_n)+2\big(\omega(\dot{x}_n)+d(x_n)/2\big)\dot{t}_n
	-\beta(x_n)\dot{t}_n^{2} \\
	=L_0(x_n, \dot{x}_n)
	+ \frac{\big(\omega(\dot{x}_n)+d(x_n)/2\big)^2-C_{z_n}^{2}}{\beta(x_n)},
	\end{multline*}
	therefore, using also \eqref{eq:C_zn1} we obtain
	\begin{multline}
		\label{eq:int_k_Ln}
		\int_{a_{k-1}}^{a_{k}} L\big((x_n,t_n), (\dot{x}_n,\dot{t}_n)\big)\de s=
		\int_{a_{k-1}}^{a_{k}} L_0(x_n, \dot{x}_n) \diff s\\
		+\int_{a_{k-1}}^{a_{k}}
		\frac{\big(\omega(\dot{x}_n)+d(x_n)/2\big)^2}{\beta(x_n)} \diff s 
		-\frac{1}{b_n^k}
		\left(
			\int_{a_{k-1}}^{a_{k}}
			\frac{\omega(\dot{x}_n)+d(x_n)/2}{\beta(x_n)}\diff s
		\right)^2\\
		+ \frac{2T_n^k}{b_n^k}
		\int_{a_{k-1}}^{a_{k}}
		\frac{\omega(\dot{x}_n)+d(x_n)/2}{\beta(x_n)}\diff s
		- \frac{(T_n^k)^{2}}{b_n^k}.
	\end{multline}
	By the Schwartz inequality in $L^2$, we obtain
	\begin{align*}
		\lefteqn{\left(
				\int_{a_{k-1}}^{a_{k}}
				\frac{\omega(\dot{x}_n)+d(x_n)/2}{\beta(x_n)}\diff s
		\right)^2}&\\
					 &\quad\quad\le\left(
						 \int_{a_{k-1}}^{a_{k}}
						 \frac{\diff s}{\beta(x_n)}
					 \right)
					 \int_{a_{k-1}}^{a_{k}}
					 \frac{\big(\omega(\dot{x}_n)+d(x_n)/2\big)^2}{\beta(x_n)}\diff s\\
					 &\quad\quad= 
					 b_n^k \int_{a_{k-1}}^{a_{k}}
					 \frac{\big(\omega(\dot{x}_n)+d(x_n)/2\big)^2}{\beta(x_n)}\diff s.
	\end{align*}
	Hence, from \eqref{eq:int_k_Ln} we obtain
	\begin{multline}
		\label{Jcoercive}
		\int_{a_{k-1}}^{a_{k}}L\big((x_n,t_n), (\dot{x}_n,\dot{t}_n)\big)\de s \ge
		\int_{a_{k-1}}^{a_{k}} L_0(x_n, \dot{x}_n) \diff s
		\\+ \frac{2T_n^k}{b_n^k}
		\int_{a_{k-1}}^{a_{k}}
		\frac{\omega(\dot{x}_n)+d(x_n)/2}{\beta(x_n)}\diff s
		- \frac{(T_n^k)^{2}}{b_n^k}.
	\end{multline}
	Since $L_0$ is the Lagrangian in a stationary product type local structure,
	as for \eqref{quadgrowth}, we deduce that there exist two positive constants
	$\ell_1,\ell_2 \in \mathbb{R}$ such that, for all the domains $U_k$ of the charts,
	we have 
	\[
		L_0(x_n,\dot{x}_n) \ge \ell_1 \norm{\dot{x}_n}^2 - \ell_2.
	\]
	Since $d_k$, $T_n^k$ and $1/b_{n}^k$ are bounded for each $k$, 
	we obtain the existence of two positive constants 
	$E_1,E_2$ (depending on $\nu,\mu,\Delta,D_0, D_1,\ell_2$)
	such that
	\begin{multline*}
		c \ge \J(z_n)
		= \int_{0}^1 L(x_n, \dot{z}_n)\diff s
		= \sum_{k = 1}^N \int_{a_{k-1}}^{a_{k}}
		L\big((x_n,t_n), (\dot{x}_n,\dot{t}_n)\big) \diff s\\
		\ge
		\ell_1 \int_{0}^{1} \norm{\dot{x}_n}^2\diff s
		- E_1 \int_{0}^1 \norm{\dot{x}_n}\diff s
		- E_2.
	\end{multline*}
	As a consequence, \eqref{eq:Fxn-bounded} holds
	and by \eqref{eq:C_zn1} we conclude that
	$\mathcal{N}_{p,q}$ is $c$-precompact.
\end{proof}

\begin{remark}\label{compactset}
	If $\mathcal{N}_{p,q}$ is $c$-precompact,
	then there exists a compact subset of $M$ 
	that contains the images of all curves in $\J^c$.
	Therefore, Assumption \ref{ass:bounds} holds on such a compact set.
\end{remark}

From Theorem~\ref{teo:cprecbound}, Remark~\ref{compactset}, Theorem~\ref{teo:main}, Theorem~\ref{teo:multiplicity} we deduce the following corollary.
\begin{corollary}
	\label{cor:precompact-boundedBelow}
	Let $L\colon TM \to \mathbb{R}$ satisfy
	Assumptions~\ref{ass:L} and \ref{ass:L1}.
	If $\mathcal{N}_{p,q}$ is $c$-precompact
	for some $c\in\R$ such that $\J^c\neq \emptyset$,
	then $\J^c$ is bounded from below and it admits a minimizer which is
	critical point of $\A$.
	Moreover, if  $\J$ is pseudocoercive, $K$ is complete, 
	and $M$ is a non-contractible manifold,
	then $\mathcal N_{p,q}\neq \emptyset$ and there exists a sequence $(z_n)_{n\in\mathbb{N}}\subset \Omega_{p,q}^{1,2}$
	of critical points of $\A$ such that
	$\lim_{n\to\infty}\A(z_n) = + \infty$.
\end{corollary}
Recalling Example~\ref{statlor}, by Corollary~\ref{cor:precompact-boundedBelow}  we then obtain the following extension of \cite[Theorems 1.2 and 1.3]{giannoni1999} to $C^1$ stationary Lorentzian manifolds. 
\begin{corollary}
	Let $(M,g)$ be a Lorentzian manifold such that $g$ is a $C^1$ metric endowed with a timelike Killing vector field $K$. If $\mathcal{N}_{p,q}$ is $c$-precompact
	for some $c\in\R$ such that $\J^c\neq \emptyset$,
	then there exists a geodesic connecting $p$ to $q$.
	Moreover, if  $\J$ is pseudocoercive, $K$ is complete, 
	and $M$ is a non-contractible manifold,
	then $\mathcal N_{p,q}\neq \emptyset$ and there exists a sequence of geodesics  $(z_n)_{n\in\mathbb{N}}\subset \Omega_{p,q}^{1,2}$
	with unbounded energy.
\end{corollary}

\begin{remark}
	Apart from completeness of $K$, whenever $d=0$, a condition ensuring that
	$\mathcal N_{p,q}$ is non-empty for all $p$ and $q$ in $M$ is that the
	distribution $\mathcal D$ defined by the kernel of $Q$ is not integrable
	through any point in $M$.
	Indeed by Chow-Rashevskii Theorem,
	there exists then a horizontal $C^1$ curve $\gamma$ connecting $p$ to $q$.
	Hence, such curve  belongs to $\mathcal N_{p,q}$ with constant $Q(\dot\gamma)=0$.
	We recall that, in the case when $L$ is the quadratic form associated with a
	stationary Lorentzian metric $g_L$ with Killing vector field $K$,
	the non-integrability of $\mathcal D$ through any point is equivalent to the fact
	that $K$ is not static in any region of $M$.
	Geodesic connectedness of a smooth static Lorentzian manifold
	was studied in \cite{CaMaPi03};
	we point out that, thanks to Theorem~\ref{CDM},
	the results in \cite{CaMaPi03} can be extended to a $C^1$ static Lorentzian metric.   
\end{remark}

\section{Dynamic  conditions for pseudocoercivity} 
Inspired by Appendix A in \cite{giannoni1999}, we give some conditions that
ensure that $\mathcal J$ is pseudocoercive.

Let us assume that there exists a $C^1$
function $\varphi\colon M\to \R$ which satisfies the monotonicity condition
$\de \varphi(K)>0$.

If $K$ is complete,
this implies that $M$ is foliated by level sets of the function $\varphi$, and
it splits as $\Sigma\times \R$, where $\Sigma$ is one of this level set.
Notice that \cite[Assumption (4.11)]{giannoni1999}
implies the completeness of the timelike Killing vector
field there, so the setting leading to \cite[Proposition A.3]{giannoni1999} is
actually analogous to ours (compare also with \cite[Theorem 2.3]{CaFlSa08}).
Some differences, on the other hand, are that the splitting $\Sigma\times\R$ is
only $C^1$ and there is no simple link between convexity properties of the
induced Lagrangian $L_0$ and the level set $\Sigma$ (see Remark~\ref{L0sigma}).  

Since   $\Sigma$ is transversal to $K$, using  Assumption~\ref{ass:L} and
arguing as in the proof of Proposition~\ref{productype},
we get that $L$ is given by \eqref{eq:gen-L} in $\Sigma\times \R$
for a $C^1$ Lagrangian
$L_0\colon T\Sigma\to\R$.
Let us denote by $g_\Sigma$ the $C^1$  Riemannian metric
on $\Sigma$ induced by $g$. We assume that the one-form $\omega$ induced by $Q$
on $\Sigma$ has sublinear growth w.r.t. the distance $d_{\Sigma}$ induced by
$g_\Sigma$, i.e. there exist  $\alpha\in [0,1)$ and two non-negative constants
$k_0$ and $k_1$ such that
\begin{equation}
	\label{sublinear}
	\|\omega\|_{\Sigma}\leq k_0+k_1\big(d_\Sigma(x,x_0)\big)^\alpha,
\end{equation}
for some $x_0\in \Sigma$ and all $x\in \Sigma$.
We recall that $\beta$ in an  expression like \eqref{eq:gen-L}
for $L$ is equal to $-Q(K)/2$ (see \eqref{asplitting}).
\begin{proposition}\label{cbounded}
	Let $L$ satisfy Assumption~\ref{ass:L} with $d$ in \eqref{noether} bounded
	and $K$ complete.
	Let $\varphi\colon M\to \R$ be a $C^1$ function such that $\de \varphi(K)>0$.
	Let $\Sigma$ a level set of $\varphi$ and $L_0$ be the
	Lagrangian induced by $L$ on $\Sigma$.
	Assume that
	\begin{itemize}
		\item $L_0$ satisfies \eqref{L0} and \eqref{L0convex}
			in Example~\ref{ex:general-ind-Lagrangian}
			(namely it satisfies the growth conditions and the pointwise convexity)
			and there exist three constants  $\ell_1, \ell_2, \ell_3$ such that
			$\lambda_0(x)\geq  \ell_1>0$,
			$L_0(x,0)\geq \ell_2$ and $\|\partial_vL_0(x,0)\|_{\Sigma}\leq \ell_3$;
		\item $\omega$ satisfies \eqref{sublinear};
		\item there exist two constant $b_1$ and $b_2$ such that $0<b_1\leq \beta(x)\leq b_2$, for all $x\in \Sigma$.
	\end{itemize}
	Then $\J$ is pseudocoercive.
\end{proposition}
\begin{proof}
	Recalling that,  by Definition~\ref{cprecompact}, $\J$ is pseudocoercive
	if $\mathcal{N}_{p,q}$ is $c$-precompact
	for all $c \in \mathbb{R}$,
	the thesis follows from Proposition~\ref{prop:cboundTOcprecompact}
	by showing that 
	that $\mathcal N_{p,q}$ is $c$-bounded for all $c\in \R$.

	Let us set $\Delta:=t(q)-t(p)$ and let 
	$z_n=z_n(s)=\big(x_n(s), t_n(s)\big)\in \J^c$
	be a sequence such that $|N(z_n,\dot z_n)|\to\sup_{z\in\J^c} |N(z,\dot z)|$.
	As for \eqref{Jcoercive}, we then get
	\begin{multline}
		\label{sublinear1}
		c\geq \int_0^1  L\big((x_n,t_n), (\dot{x}_n,\dot{t}_n)\big)\de s\\ \ge
		\int_0^1 L_0(x_n, \dot x_n) \diff s
		+ \frac{2\Delta}{b_n}
		\int_0^1
		\frac{\omega(\dot x_n)+d(x_n)/2}{\beta(x_n)}\diff s
		- \frac{\Delta^2}{b_n},
	\end{multline}
	where $b_n=\int_0^1\frac{1}{\beta(x_n(s))}\de s$.
	Then  taking into account that $d$ is bounded,
	$0<b_1\leq \beta(x) \leq  b_2$, $\lambda_0(x)\geq  \ell_1>0$,  $L_0(x,0)\geq
	\ell_2$ and $\|\partial_vL_0(x,0)\|_{\Sigma}\leq \ell_3$,
	for all $x\in \Sigma$,
	using \eqref{quadgrowth} for $L_0$ and \eqref{sublinear},
	we obtain from \eqref{sublinear1} that
	$\int_0^1\|\dot x_n\|^2_\Sigma\de s$ is bounded.
	Analogously to \eqref{eq:tn_Cn}	we have then
	\[
		N(z_n,\dot z_n)=2C_{z_n}=
		\frac{2}{b_n}\left(\int_0^1\frac{\omega(\dot x_n)+d(x_n)/2}{\beta}\de s-\Delta\right)
	\]
	and hence $N(z_n, \dot z_n)$ is bounded as well.  
\end{proof}
\begin{remark}
	The proof of  Proposition~\ref{cbounded}
	also shows that the manifold $\mathcal N_{p,q}$ associated to
	the Lagrangian in Example~\ref{ex:general-ind-Lagrangian}
	is $c$-bounded for all $c\in\R$ provided that $L_0$ satisfies \eqref{L0} and
	\eqref{L0convex} in Example~\ref{ex:general-ind-Lagrangian}, $d$ is bounded,
	$\omega$ has sublinear growth on $S$
	(hence \eqref{sublinear} holds)
	and there exist some constants
	$b_1, b_2, \ell_1, \ell_2, \ell_3$ such that
	$0<b_1\leq \beta(x)\leq b_2$, $\lambda_0(x)\geq  \ell_1>0$,  $L_0(x,0)\geq \ell_2$
	and $\|\partial_vL_0(x,0)\|_{\Sigma}\leq \ell_3$, for all $x\in S$.
\end{remark}

\begin{remark}\label{L0sigma}
	The strong convexity condition for $L_0$ holds if $L_0$,
	satisfying \eqref{monotone} on $T\Sigma$, satisfies also
	\eqref{strong} on $\Sigma$.
	This condition can be considered as a replacement of being $\Sigma$ a
	spacelike and complete hypersurface when $L$ is the quadratic form of a
	stationary Lorentzian manifold
	(in our setting the Riemannian metric on $\Sigma$, induced by the auxiliary one $g$,
	is complete because $g$ is complete by assumptions).
	Indeed, in such a case, it is enough to assume that
	$\nabla\varphi$ is timelike (i.e. $\varphi$ is a $C^1$ time function)
	to get that a level set $\Sigma$ of $\varphi$ is spacelike.
	The existence of such a $\varphi$ is guaranteed if there exists a spacelike hypersurface
	that intersects once every flow line of the complete timelike Killing vector field
	$K$ (see \cite[Appendix A]{giannoni1999}).
\end{remark}

\appendix
\section{Regularity of the Critical Points}
\label{sec:regularity}
In this section we show that a critical point of the action functional $\A$ on
$\Omega_{p,q}$ is actually a curve of class $C^1$. This is a quite standard result in relation with Assumptions \eqref{ass:L} and \eqref{ass:L1}, but we give the details for the reader convenience.
\begin{proposition}
	Let $L\colon TM \to \mathbb{R}$ be a Lagrangian
	satisfying Assumptions~\ref{ass:L} and \ref{ass:L1}
	and let $z$ be a critical point of the action functional
	$\A\colon \Omega^{1,2}_{p,q}\to\R$, $\A(z)=\int_0^1L(z,\dot z)\diff s$.
	Then, both $z$ and $\partial_v L(z,\dot{z})$ are of class $C^1$,
	the Euler-Lagrange equation \eqref{eq:Euler-Lagrange} holds
	in  classical sense, namely for all $i\in\{0,\ldots,m\}$,
	\begin{equation}
		\label{ELagain}
		\frac{\partial L}{\partial z^i}\big(z(s),\dot{z}(s)\big)
		= \frac{d}{ds}
		\left(\frac{\partial L}{\partial v^i}\big(z(s),\dot{z}(s)\big)\right),
		\quad \text{for all $s\in [0,1]$,}
	\end{equation}
	and $z$ satisfies the conservation law
	\begin{equation}
		\label{conslaw}
		\partial_v L(z,\dot{z})[\dot{z}] - L(z,\dot{z}) = E,
	\end{equation}
	for some constant $E\in\R$. 
\end{proposition}
\begin{proof}	 
	As regularity of a critical curve is a local result,
	by Proposition \ref{productype} we can assume, without loosing generality,
	that $L$ is a Lagrangian on $U\times I$,
	where $U$ is a precompact open neighborhood of $\R^m$ and $I\subset\R$
	an open interval, defined as
	\begin{equation}
		\label{againlocal}
		L\big((x,t),(\nu,\tau)\big)
		=L_0(x,\nu) + 2\big(\omega(\nu)+d(x_n)/2\big)\tau -\beta(x)\tau^{2},
	\end{equation}
	for all $\big((x,t),(\nu,\tau)\big)\in (U\times I)\times (\R^m\times\R)$.
	Arguing as in the proof of Proposition~\ref{productype},
	for any point $\bar z \in M$,
	we can take  $U$ as a hypersurface in $M$ passing through $\bar z$
	such that $\omega$ vanishes at $\bar z$.
	Let $z:[0,1]\to U\times I$, $z(s)=(x(s), t(s))$ be a critical point for $\A$
	then for all $(\xi,\eta)\in H^1_0([0,1], \R^m)\times H^1_0([0,1], \R)$ we have 
	\begin{multline}
		\label{eq:diff-fab}
		0=	\diff \A(z)[(\xi,\eta)]
		=
		\inte
		\left(
			\partial_x L_0(x,\dot{x})[\xi]+
		\partial_\nu L_0(x,\dot{x})[\dot{\xi}]\right)\diff s\\
		+ 2
		\inte
		\left(\partial_x\omega(\xi,\dot x)\dot{t}
			+ \omega(\dot \xi)\dot{t} + \omega(\dot{x})\dot{\eta}
			+\frac{1}{2}\de d(\xi)\dot t
			+\frac{d(x)}{2}\dot \eta
		\right)\diff s\\
		- 
		\inte
		\left(
			\diff\beta(\xi)\dot{t}^2 + 
			2\beta(x)\dot{t}\dot{\eta}
		\right)\diff s.
	\end{multline}
	Since $z=(x,t)$ is a critical point of $\A$, 
	there exists a constant $C_z\in \mathbb{R}$ such that
	\[
		N(z,\dot z) =2\omega(\dot{x})- 2 \beta(x)\dot{t}+d(x)=2C_z,
	\]
	hence we have
	\begin{equation}
		\label{eq:dott-Cz}
		\dot{t} = \frac{\omega(\dot{x})+d(x)/2 - C_z}{\beta(x)}.
	\end{equation}
	Moreover, from \eqref{eq:diff-fab},
	for all $\xi \in C^{\infty}_{0}([0,1],\mathbb{R}^m)$ we have
	\begin{multline*}
		\diff A(z)[\xi,0]
		=
		\inte
		\left(
			\partial_x L_0(x,\dot{x})[\xi]+
		\partial_\nu L_0(x,\dot{x})[\dot{\xi}]\right)\diff s\\
		+ 2
		\inte
		\left(\partial_x\omega(\xi,\dot{x})
			+ \omega(\dot{\xi})
		+\frac{1}{2}\de d(\xi)\right)\dot{t}\diff s
		- 
		\inte
		\diff\beta(\xi)\dot{t}^2 \diff s
		= 0,
	\end{multline*}
	hence, 
	\begin{multline*}
		\inte
		\left(
			\partial_\nu L_0(x,\dot{x})[\dot{\xi}] + 
			2	\omega(\dot{\xi})\dot{t}
		\right) \diff s = 
		- \inte			
		\partial_x L_0(x,\dot{x})(\xi)\diff s\\
		-\inte \big(2\partial_x\omega(\xi,\dot x)\dot{t}
			+\de d(\xi)\dot t 
			-\diff\beta(\xi)\dot{t}^2 
		\big) \diff s.
	\end{multline*}
	Then, there exists an $L^1$ map $h\colon [0,1] \to (\mathbb{R}^m)^*$ such that
	\[
		\inte
		\left(
			\partial_\nu L_0(x,\dot{x})[\dot{\xi}] + 
			2	\omega(\dot{\xi})\dot{t}
		\right) \diff s = 
		\inte h(s)[\xi]\diff s.
	\]
	Denoting by $H$ a primitive of $- h$, 
	we obtain the existence of a constant $A \in (\mathbb{R}^m)^*$
	such that 
	\begin{equation}
		\label{eq:A5bis} 
		\partial_\nu L_0(x,\dot{x}) + 
		2	\dot{t}\omega_x
		= 
		A + H
		\qquad \text{a.e. on }[0,1].
	\end{equation}
	Using \eqref{eq:dott-Cz}
	we obtain
	\begin{equation}
		\label{eq:smooth-proof2}
		\partial_\nu L_0(x,\dot{x}) + 
		2\frac{\omega(\dot x)}{\beta(x)}\omega_x
		= A + H + \frac{2C_z-d(x)}{\beta(x)}\omega_x,
		\qquad \text{a.e. on }[0,1],
	\end{equation}
	where the right-hand side is an absolute continuous function.

	Let us consider the continuous maps
	$\mathcal L\colon U\times \mathbb{R}^m \to (\mathbb{R}^m)^*$
	and 
	$\mathcal{P}\colon U\times\mathbb{R}^m \to U\times (\mathbb{R}^m)*$
	defined respectively  as
	\[
		\mathcal L(x,\nu): = \partial_\nu L_0(x,\nu) + 2
		\frac{\omega(\nu)}{\beta(x)}\omega_x
	\]
	and 
	\[
		\mathcal P(x,\nu) := \left(x,\mathcal{L}(x,\nu)\right).
	\]
	As in the proof of Proposition~\ref{productype},
	recalling how $U$ has been chosen, and  up to take a smaller $U$,
	we can state that there exists $C>0$ such that for any $x\in U$ and for all
	$\nu_1,\nu_2\in \R^m$
	\begin{equation}
	\label{monotoneagain}
	\big(\mathcal L(x,v_2)-\mathcal L(x,v_1)\big)[v_2-v_1]\geq C \|v_2-v_1\|^2.
	\end{equation}
	Notice that \eqref{monotoneagain} implies that for each $x\in U$,
	$\mathcal L(x,\cdot)$ is injective
	with inverse which is continuous on the image of $\mathcal L(x,\cdot)$.
	Using again \eqref{monotoneagain} together with the continuity of $\mathcal L$
	on $U\times\R^m$,
	we get that the map $\mathcal P$ is injective with continuous inverse as well.
	Hence,by \eqref{eq:smooth-proof2}, 
	\begin{equation}
		\label{bootstrap}
		\big(x(s),\dot{x}(s)\big)=
		\mathcal P^{-1}\bigg(x(s),
			A + H(s) + \frac{2C_z-d(x)}{\beta(x(s))}\omega_{x(s)}
		\bigg)
	\end{equation}
	and so $x$ is of class $C^1$.
	By \eqref{eq:dott-Cz}, even $\dot{t}$ is continuous, so
	$z$ is of class $C^1$ in the coordinate system where the stationary product
	type local structure \eqref{againlocal} of $L$ holds and then in any other
	coordinate system.
	Hence, the function $h$ is actually a continuous function, 
	and the right-hand side of 
	\eqref{eq:A5bis} is of class $C^1$.
	Since $\omega$ is a $C^1$ one-form, \eqref{eq:smooth-proof2}
	shows that also $\partial_\nu L_0(x,\dot{x})$ is of class $C^1$.
	Being $\partial_vL(z,\dot z)$ identifiable with 
	\[
		\Big(
			\partial_\nu L_0(x,\dot{x}) 
			+ 2 \omega \dot{t}
			, \big(2\omega(\dot x)+d(x)-2\beta(x)\dot t\big)\de t
		\Big),
	\]
	we deduce that the function
	$s\in[0,1]\mapsto \partial_vL\big(z(s),\dot z(s)\big)$
	is $C^1$ as well, and then \eqref{ELagain} holds.
	Moreover, by standard arguments, 
	$z$ satisfies the conservation law \eqref{conslaw} 
	for some constant $E \in \mathbb{R}$ (see, e.g., Proposition 1.16 of \cite{Buttazzo1998}).
\end{proof}
\begin{remark}
	We notice that,
	if $L_0$ admits positive definite  vertical Hessian at some  vector in $\nu\in TU$,
	then $\mathcal L$ admits a bijective fiberwise derivative,
	so it is a local $C^1$-diffeomorphism in
	a neighborhood $\mathcal V$ of $\nu$ in $TU$.
	Hence, $\mathcal P$ has a $C^1$ inverse on
	$\mathcal P(\mathcal V)$ and then from \eqref{bootstrap} we get that
	$\dot x$ is $C^1$ on an open interval $J$
	containing the instant $s_0$ such that $\dot x(s_0)=\nu$.
	From \eqref{eq:dott-Cz}, $\dot t$ is $C^1$ as well on $J$
	and then $z\in C^2(J,M)$.
	We observe  that this holds in particular when $L$ is the quadratic form
	associated with  $C^1$ stationary Lorentzian metric $g_L$ (see Example~\ref{statlor}),
	hence its critical curves are $C^2$
	on the interval where they are defined and then they are classical geodesics.			
\end{remark}


\begin{thebibliography}{99}

	\bibitem{AazJav16}
	{
	A.~B. Aazami and M.~A. Javaloyes}:
	{
	Penrose's singularity theorem in a Finsler spacetime}.
	Classical  Quantum Gravity,
	33 (2016), p.~025003.
	\href{https://doi.org/10.1088/0264-9381/33/2/025003}{https://doi.org/10.1088/0264-9381/33/2/025003}


	\bibitem{Abbondandolo2007}
	{
	A.~Abbondandolo and A.~Figalli}:
	{
	High action orbits for {{Tonelli Langrangians}} and superlinear {{Hamiltonians}} on compact configuration spaces}.
	J. Differential Equations,
	234 (2007), pp.~626--653.
	\href{https://doi.org/10.1016/j.jde.2006.10.015}{https://doi.org/10.1016/j.jde.2006.10.015}

	\bibitem{Abbondandolo2009}
	{
	A.~Abbondandolo and M.~Schwarz}:
	{
	A smooth pseudo-gradient for the Lagrangian action functional}.
	Adv. Nonlinear Stud., 9 (2009), pp.~597--623.
	\href{https://doi.org/10.1515/ans-2009-0402}{https://doi.org/10.1515/ans-2009-0402}

	\bibitem{avez1963}
	{
	A.~Avez}:
	{
	{Essais de g\'eom\'etrie riemannienne hyperbolique globale.  Applications \`a la relativit\'e g\'en\'erale}}.
	Ann. Inst. Fourier,
	13 (1963), pp.~105--190.
	\href{https://doi.org/10.5802/aif.144}{https://doi.org/10.5802/aif.144}

	\bibitem{BaChSh00}
	{
	D.~Bao, S.-S. Chern and Z.~Shen}:
	{
	{An Introduction to {R}iemann-{F}insler geometry}}.
	{Graduate Texts in Mathematics},
	Springer-Verlag, New York, 2000.

	\bibitem{Bartol99}
	{
	R.~Bartolo}:
	{
		Trajectories connecting two events of a Lorentzian manifold
	in the presence of a vector field}.
	J. Differential Equations,
	153 (1999), pp.~82--95
	\href{https://doi.org/10.1006/jdeq.1998.3521}{https://doi.org/10.1006/jdeq.1998.3521}

	\bibitem{Beem70}
	{
	J.~K. Beem}:
	{
	Indefinite {F}insler spaces and timelike spaces}.
	Canad. J. Math.,
	22 (1970), pp.~1035--1039.
	\href{https://doi.org/10.4153/CJM-1970-119-7}{https://doi.org/10.4153/CJM-1970-119-7}

	\bibitem{Benci86}
	{
	V.~Benci}:
	{
	{Periodic solutions of {L}agrangian systems on a compact manifold}}.
	J. Differential Equations,
	63 (1986), pp.~135--161.
	\href{https://doi.org/10.1016/0022-0396(86)90045-8}{https://doi.org/10.1016/0022-0396(86)90045-8}


	\bibitem{benci1991}
	{
	V.~Benci, D.~Fortunato and F.~Giannoni}:
	{
	On the existence of multiple geodesics in static space-times}.
	Ann. Inst. H. Poincar\'e C Anal. Non Lin\'eaire
	8 (1991), pp.~79--102.
	\href{https://doi.org/10.1016/S0294-1449(16)30278-5}{https://doi.org/10.1016/S0294-1449(16)30278-5}

	\bibitem{BeJaSa20}
	{
	A.~N. Bernal, M.~A. Javaloyes, and M.~Sánchez}:
	{
	Foundations of Finsler spacetimes from the observers’ viewpoint}.
	Universe,
	6 (2020), 55.
	\href{https://doi.org/10.3390/universe6040055}{https://doi.org/10.3390/universe6040055}

	\bibitem{BerSuh18}
	{
	P.~Bernard and S.~Suhr}:
	{
	Lyapounov functions of closed cone fields: from Conley theory to time functions}.
	Comm. Math. Phys.,
	359 (2018), pp.~467--498.
	\href{https://doi.org/10.1007/s00220-018-3127-7}{https://doi.org/10.1007/s00220-018-3127-7}

	\bibitem{Bolza04}
	O.~Bolza: Lectures on the Calculus of Variations. {University} of {Chicago} {Press}, Chicago, 1904.
	\href{http://www.hti.umich.edu/cgi/t/text/text-idx?c=umhistmath;idno=ACM2513}{http://www.hti.umich.edu/cgi/t/text/text-idx?c=umhistmath;idno=ACM2513}

	\bibitem{Brandt92}
	{
	H.~E. Brandt}:
	{
	Finsler-spacetime tangent bundle}.
	Foundations of Physics Letters,
	5 (1992), pp.~221--248.
	\href{https://doi.org/10.1007/BF00692801}{https://doi.org/10.1007/BF00692801}

	\bibitem{Buttazzo1998}
	{
	G.~Buttazzo, M.~Giaquinta and S.~Hildebrandt}:
	{
	One-Dimensional Variational Problems: An Introduction}.
	The Clarendon Press Oxford University Press, New York, 1998.

	\bibitem{CaFlSa08}
	{
	A.~M. Candela, J.~L. Flores and M.~S{\'a}nchez}:
	{
		Global hyperbolicity and Palais-Smale condition for action functionals in
	stationary spacetimes}.
	Adv. Math.,
	218 (2008), pp.~515--556.
	\href{https://doi.org/10.1016/j.aim.2008.01.004}{https://doi.org/10.1016/j.aim.2008.01.004}
	\bibitem{CapMas02}
	{
	E.~Caponio and A.~Masiello}:
	{
	Trajectories for relativistic particles under the action of an electromagnetic field in a stationary space-time}.
	Nonlinear Anal. Theory Meth. Appl.,
	50 (2002), p.~71--89.
	\href{https://doi.org/10.1007/978-88-470-2101-3_28}{https://doi.org/10.1007/978-88-470-2101-3\_28}

	\bibitem{CapMas20}
	{
	E.~Caponio and A.~Masiello}:
	{
	On the analyticity of static solutions of a field equation in Finsler gravity}.
	Universe,
	6 (2020), 59.
	\href{https://doi.org/10.3390/universe6040059}{https://doi.org/10.3390/universe6040059}

	\bibitem{CaMaPi03}
	{
	E.~Caponio, A.~Masiello and P.~Piccione}:
	{
	Some global properties of static spacetimes}.
	Math. Z., 
	244 (2003), pp.~457--468.
	\href{https://doi.org/10.1007/s00209-003-0488-0}{https://doi.org/10.1007/s00209-003-0488-0}

	\bibitem{CaMaPi04}
	{
	E.~Caponio, A.~Masiello and P.~Piccione}: Maslov index and {M}orse theory for the relativistic {L}orentz force equation. 
	Manuscripta Math., 113 (2004), pp.~471--506.
	\href{https://doi.org/10.1007/s00229-004-0441-5}{https://doi.org/10.1007/s00229-004-0441-5}


	\bibitem{CapSta16}
	{
	E.~Caponio and G.~Stancarone}:
	{
	Standard static Finsler spacetimes}.
	Int. J. Geom. Methods Mod. Phys.,
	13 (2016), 1650040.
	\href{https://doi.org/10.1142/S0219887816500407}{https://doi.org/10.1142/S0219887816500407}

	\bibitem{caponio2018}
	{
	E.~Caponio and G.~Stancarone}:
	{
	On Finsler spacetimes with a timelike Killing vector field}.
	Classical Quantum Gravity,
	35 (2018), 085007.
	\href{https://doi.org/10.1088/1361-6382/aab0d9}{https://doi.org/10.1088/1361-6382/aab0d9}

	\bibitem{corvellec1993}
	{
	J.-N. Corvellec, M.~Degiovanni and M.~Marzocchi}:
	{
	Deformation properties for continuous functionals and critical point theory}.
	Topol. Methods Nonlinear Anal.,
	1 (1993), pp.~151--171.
	\href{https://doi.org/10.12775/TMNA.1993.012}{https://doi.org/10.12775/TMNA.1993.012}


\bibitem{CraMes08}
	{M. Crampin and T.Mestdag}:
	Routh's procedure for non-abelian symmetry groups.
	J. Math. Phys., 49 (2008), 032901.
	\href{https://doi.org/10.1063/1.2885077}{https://doi.org/10.1063/1.2885077}


	\bibitem{Degiovanni1994}
	{
	M.~Degiovanni and M.~Marzocchi}:
	{
	A critical point theory for nonsmooth functional}.
	Ann. Mat. Pura Appl.,
	167 (1994), pp.~73--100.
	\href{https://doi.org/10.1007/BF01760329}{https://doi.org/10.1007/BF01760329}

	\bibitem{Ekelan74}
	{
	I.~Ekeland}:
	{
	On the variational principle}.
	J. Math. Anal. Appl.,
	47 (1974), pp.~324--353.
	\href{https://doi.org/10.1016/0022-247X(74)90025-0}{https://doi.org/10.1016/0022-247X(74)90025-0}


	\bibitem{fadell1991}
	{
	E.~Fadell and S.~Husseini}:
	{
	Category of loop spaces of open subsets in euclidean space}.
	Nonlinear Anal. Theory Meth. Appl.,
	17 (1991), pp.~1153--1161.
	\href{https://doi.org/10.1016/0362-546X(91)90234-R}{https://doi.org/10.1016/0362-546X(91)90234-R}


	\bibitem{FatSic12}
	{
	A.~Fathi and A.~Siconolfi}:
	{
	On smooth time functions}.
	Math. Proc. Cambridge Philos. Soc.,
	152 (2012), pp.~303--339.
	\href{https://doi.org/10.1017/S0305004111000661}{https://doi.org/10.1017/S0305004111000661}

	\bibitem{Fonseca07}
	{
	I.~Fonseca and G.~Leoni}:
	{
	Modern Methods in Calculus of Variations. $L^p$ Spaces}.
	Springer New York, NY, 2007
	\href{https://doi.org/10.1007/978-0-387-69006-3}{https://doi.org/10.1007/978-0-387-69006-3}

	\bibitem{GaPiVi12}
	{
	R.~Gallego~Torrom\'{e}, P.~Piccione and H.~Vit\'{o}rio}:
	{
	On Fermat's principle for causal curves in time oriented Finsler spacetimes}.
	J. Math. Phys., 53 (2012), 123511.
	\href{https://doi.org/10.1063/1.4765066}{https://doi.org/10.1063/1.4765066}

	\bibitem{GiannoniMasiello1991}
	{
	F.~Giannoni and A.~Masiello}:
	{
	On the existence of geodesics on stationary Lorentz manifolds with convex boundary}.
	J. Funct. Anal.,
	101 (1991), pp.~340--369.
	\href{https://doi.org/10.1016/0022-1236(91)90162-X}{https://doi.org/10.1016/0022-1236(91)90162-X}

	\bibitem{giannoni1999}
	{
	F.~Giannoni and P.~Piccione}:
	{
	An intrinsic approach to the geodesical connectedness of stationary Lorentzian manifolds}.
	Comm. Anal. Geom.,
	7 (1999), pp.~157--197.
	\href{https://dx.doi.org/10.4310/CAG.1999.v7.n1.a6}{https://dx.doi.org/10.4310/CAG.1999.v7.n1.a6}


	\bibitem{HasPer19}
	{
	W.~Hasse and V.~Perlick}:
	{
	Redshift in Finsler spacetimes}.
	Phys. Rev. D,
	100 (2019), 024033.
	\href{https://doi.org/10.1103/PhysRevD.100.024033}{https://doi.org/10.1103/PhysRevD.100.024033}

	\bibitem{HoPfVo19}
	{
	M.~Hohmann, C.~Pfeifer and N.~Voicu}:
	{
	Finsler gravity action from variational completion}.
	Phys. Rev. D,
	100 (2019), 064035.
	\href{https://doi.org/10.1103/PhysRevD.100.064035}{https://doi.org/10.1103/PhysRevD.100.064035}

	\bibitem{Horvat58}
	{
	J.~I. Horv\'{a}th}:
	{
	New geometrical methods of the theory of physical fields}.
	Il Nuovo Cimento,
	9 (1958), pp.~444--496.
	\href{https://doi.org/10.1007/BF02747685}{https://doi.org/10.1007/BF02747685}

	\bibitem{JavSan20}
	{
	M.~A. Javaloyes and M.~S\'{a}nchez}:
	{
	On the definition and examples of cones and Finsler spacetimes}.
	Rev. R. Acad. Cienc. Exactas F\'is. Nat. Ser. A Mat. RACSAM,
	114 (2020), 30.
	\href{https://doi.org/10.1007/s13398-019-00736-y}{https://doi.org/10.1007/s13398-019-00736-y}

	\bibitem{LaPeHa12}
	{
	C.~L{\"a}mmerzahl, V.~Perlick and W.~Hasse}:
	{
	Observable effects in a class of spherically symmetric static Finsler spacetimes}.
	Phys. Rev. D,
	86 (2012), 104042.
	\href{https://doi.org/10.1142/9789814623995_0255}{https://doi.org/10.1142/9789814623995\_0255}

	\bibitem{lang1985}
	{
	S.~Lang}:
	{
	Differential {{Manifolds}}}.
	{Springer-Verlag}, {Berlin},
	1985.

	\bibitem{LuMiOh21}
	{
	Y.~Lu, E.~Minguzzi and S.-i. Ohta}:
	{
	Geometry of weighted Lorentz-Finsler manifolds I: singularity theorems}.
	J. Lond. Math. Soc. 
	104 (2021), pp.~362--393.
	\href{https://doi.org/10.1112/jlms.12434}{https://doi.org/10.1112/jlms.12434}

	\bibitem{MaRaSc00}
	{
	J.~E. Marsden, T.~S. Ratiu and J.~Scheurle}:
	{
	Reduction theory and the Lagrange-Routh equations}.
	J. Math. Phys.
	41 (2000), pp.~3379--3429.
	\href{https://doi.org/10.1063/1.533317}{https://doi.org/10.1063/1.533317}

	\bibitem{Minguz15}
	{
	E.~Minguzzi}:
	{
	Light cones in Finsler spacetime}.
	Comm. Math. Phys.,
	(2015), pp.~1529--1551.
	\href{https://doi.org/10.1007/s00220-014-2215-6}{https://doi.org/10.1007/s00220-014-2215-6}

	\bibitem{Minguz19}
	{
	E.~Minguzzi}:
	{
	Causality theory for closed cone structures with applications}.
	Rev. Math. Phys.,
	31 (2019), 1930001.
	\href{https://doi.org/10.1142/S0129055X19300012}{https://doi.org/10.1142/S0129055X19300012}

	\bibitem{Morse28}
	M.~Morse: The foundations of a theory in the calculus of variations in the large. 
	Trans. Amer. Math. Soc., 30 (1928), pp.~213--274.
	\href{https://doi.org/10.2307/1989122}{https://doi.org/10.2307/1989122}



	\bibitem{Palais66a}
	{
	R.~S. Palais}:
	{
	Homotopy theory of infinite dimensional manifolds}.
	Topology,
	5 (1966), pp.~1--16.  
	\href{https://doi.org/10.1016/0040-9383(66)90002-4}{https://doi.org/10.1016/0040-9383(66)90002-4}

	\bibitem{Perlic06}
	{
	V.~Perlick}:
	{
	Fermat principle in Finsler spacetimes}.
	Gen. Relativity Gravitation,
	38 (2006), pp.~365--380.
	\href{https://doi.org/10.1007/s10714-005-0225-6}{https://doi.org/10.1007/s10714-005-0225-6}

	\bibitem{Rutz93}
	{
	S.~F. Rutz}:
	{
	A Finsler generalisation of Einstein's vacuum field equations}.
	Gen. Relativity Gravitation,
	25 (1993), pp.~1139--1158.
	\href{https://doi.org/10.1007/BF00763757}{https://doi.org/10.1007/BF00763757}

	\bibitem{seifert1967}
	{
	H.-J. Seifert}:
	{
	Global Connectivity by Timelike Geodesics}.
	Zeitschrift f\"ur Naturforschung A,
	22 (1967), pp.~1356--1360.
	\href{https://doi.org/10.1515/zna-1967-0912}{https://doi.org/10.1515/zna-1967-0912}

	\bibitem{Struwe08}
	{
	M.~Struwe}:
	{
		Variational Methods.
	Applications to Nonlinear Partial Differential Equations and Hamiltonian Systems}.
	Springer-Verlag, Berlin, 2008.
	\href{https://doi.org/10.1007/978-3-540-74013-1}{https://doi.org/10.1007/978-3-540-74013-1}

	\bibitem{szulkin1988}
	{
	A.~Szulkin}:
	{
	Ljusternik-Schnirelmann theory on {$C^1$}-manifolds}.
	Annales de l'I.H.P. Analyse non lin\'eaire, 5 (1988), pp.~119--139.
	\href{https://doi.org/10.1016/S0294-1449(16)30348-1}{https://doi.org/10.1016/S0294-1449(16)30348-1}

	\bibitem{Tonell25}
	L.~Tonelli:  The calculus of variations. Bull. Amer. Math. Soc., 31 (1925), pp. 163--172.
	\href{https://doi.org/10.1090/S0002-9904-1925-04002-1}{https://doi.org/10.1090/S0002-9904-1925-04002-1}

	\bibitem{Vitori20}
	{
	H.~Vit\'orio}:
	{
	On the Maslov index in a symplectic reduction and applications}.
	Proc. Amer. Math. Soc.,
	148 (2020), pp.~3517--3526.
	\href{https://doi.org/10.1090/proc/14985}{https://doi.org/10.1090/proc/14985}


\end{thebibliography}

\end{document}